\documentclass[11pt]{amsart}

\usepackage{amsmath}
\usepackage{amssymb}
\usepackage{graphicx}
\usepackage{xcolor}

\newtheorem{theorem}{Theorem}[section]

\newtheorem{lemma}[theorem]{Lemma}
\newtheorem{proposition}[theorem]{Proposition}
\theoremstyle{definition}

\newtheorem{remark}[theorem]{Remark}

\numberwithin{equation}{section}

\title[On a class of critical Schr\"odinger-Poisson systems]{On a class of critical Schr\"odinger-Poisson systems involving the $(p,q)$-Laplacian}

\author[L. Baldelli]{Laura Baldelli}
\address[L. Baldelli]{Institute for Analysis, Karlsruhe Institute of Technology (KIT) D-76128 Karlsruhe, Germany}
\email{{\tt laura.baldelli@kit.edu}}

\author[R. Filippucci]{Roberta Filippucci}
\address[R. Filippucci]{Department of Mathematics, University of Perugia, Via Vanvitelli 1, 06123 Perugia,  Italy}
\email{\tt roberta.filippucci@unipg.it}

\keywords{Schr\"odinger Poisson systems, $(p,q)$-Laplacian, Mountain pass solutions, Concentration compactness.}

\subjclass[2020]{35J20, 35B08, 35J62}

\begin{document}

\begin{abstract}
   This paper investigates a class of Schr\"odinger-Poisson systems in $\mathbb R^3$ featuring the $(p,q)$-Laplacian operator and a combination of critical and subcritical nonlinearities in the Schr\"odin\-ger equation while the $m$-Laplacian and a power type nonlinearity in the Poisson's one. We consider both the attractive and repulsive cases, which correspond to different signs in front of the nonlocal term. While most existing literature relies on auxiliary functionals or specialized techniques to overcome the lack of compactness and ensure the boundedness of Palais-Smale sequences, we employ a direct variational approach. By applying the Mountain Pass Theorem and concentration compactness principles, we establish the existence of positive solutions. A careful analysis is conducted to identify the parameter ranges for which the Mountain Pass level falls within the compactness threshold, despite the technical challenges posed by the unbalanced growth of the operator and the nonlocal interaction.
\end{abstract}

\maketitle


{\itshape Dedicated to Professor Laurent V\'eron on the occasion of  his 75th birthday}

\section{Introduction}

In this paper we investigate existence results of quasilinear Schr\"odinger-Poisson systems in $\mathbb R^3$ driven by the $(p,q)$-Laplacian operator with a critical nonlinearity perturbed by a subcritical term, namely
{ \begin{equation} \label{maineq}\tag{$\mathcal{P}_\pm$} \begin{cases}
    \mathcal L(u)\pm\lambda \phi |u|^{\vartheta-2}u=\beta \mathcal{W} |u|^{ \tau -2}u+\mathcal{K}|u|^{p^* -2}u, \,\,  &\text{in }  
    \mathbb R^3\\
    -\Delta_m\phi =|u|^{\vartheta},   &\text{in 
   } \mathbb R^3
\end{cases} \end{equation}} 
where  $\mathcal L(u):=-\Delta_pu-\Delta_qu+(|u|^{p -2}+|u|^{q -2})u$, $1<q\le p<3, \Delta_su=\mbox{div}(|\nabla u|^{s-2}\nabla u)$ is the $s$-Laplacian operator with $s>1$, $\,\lambda, \beta>0$, $1<m<3$, $p\le \tau < p^*$,  with $p^*=3p/(3-p)$ being the critical Sobolev exponent in dimension $N=3$, and 
\begin{equation}\label{defthetamain}
\max\left\{1,\frac{q}{(m^*)'}\right\}<\vartheta<\frac{p^*}{(m^*)'}.
\end{equation}
The nontrivial weights $\mathcal{K}=\mathcal{K}(x),\mathcal{W}=\mathcal{W}(x)$ satisfy
\begin{equation}\label{K}
0\le \mathcal{K}\in L^\infty(\mathbb R^3)\cap C(\mathbb R^3),
\end{equation}
\begin{equation}\label{W}
0\le \mathcal{W}\in L^{\eta}(\mathbb R^3), \quad \eta=\frac{p^*}{p^*-\tau}, \quad \mathcal{W}>0 \,\, \text{in}\,\, \Omega_\mathcal{W}\subset\mathbb R^3, \quad |\Omega_\mathcal{W}|>0.
\end{equation}

As indicated by the choice of sign in \eqref{maineq}, we investigate two distinct classes of Schr\"odinger-Poisson systems. The specific sign adopted introduces unique analytical challenges in each case. Broadly speaking, when the nonlocal term appears with a negative sign, it acts in opposition to the norm induced by the operator on the left-hand side and is, therefore, treated as a nonlinearity.
Conversely, when the nonlocal term is positive, it reinforces the norm while opposing the nonlinearity. Although both problems remain variational, the functional geometry is significantly altered by this change in sign, necessitating distinct methodological approaches that will be detailed below.

Semilinear Schr\"odinger-Poisson systems of the form 
\begin{equation} \label{Lapla}\tag{$SP_\pm$} \begin{cases} 
    -\Delta u+u\pm \phi u=f(x,u), \,\,  &\text{in} \,\, 
    \mathbb R^3,\\
    -\Delta\phi =u^2,   &\text{in }\mathbb R^3,
\end{cases} \end{equation}
describe the interaction of charged particles with an electromagnetic or gravitational field generated by the particles themselves. Here, the unknown $u$ arises from the standing wave ansatz $\psi(x,t) = e^{-it} u(x)$ for the Schr\"odinger equation where $\psi$ is the unknown, while $\phi$ represents the potential.

If the particles move in their own gravitational field, where the field is generated by the particle probability density via the classical Newton field equation, we arrive to $(SP_-)$.
Conversely, if the particles are electrically charged, the long-range electrostatic interaction can be effectively modeled by a potential term, 
we end up with $(SP_+)$, see \cite{vaira}.

In both cases, the term $\phi u$ in the first equation is nonlocal since, from the second equation, $\phi$ can be represented as the convolution of $u^2$ with the fundamental solution of the Laplacian. Finally, the nonlinear term $f(x,u)$ models the mutual interactions between particles. When $f(x,u) \equiv 0$, system \eqref{Lapla} reduces to the   so called  Schr\"odinger-Maxwell system. 

Pioneering works concerning \eqref{Lapla} with the negative sign include those by Benci and Fortunato \cite{articolo_6_1, articolo_7_1}, who investigated an eigenvalue-type problem. Subsequently, the existence of positive radial solutions for \eqref{Lapla} with superlinear and subcritical power-type nonlinearities was studied in \cite{articolopc_19}, see also \cite{aprile, aprileW} for investigations into the case $\lambda \to 0$. Due to the variational structure of \eqref{Lapla}, the aforementioned papers employ a variational approach, so that solutions are obtained as critical points of the associated energy functional. Regarding non-radial solutions, we highlight the work of Azzollini and Pomponio \cite{AP}, in which the existence of a ground state solution is established via a concentration-compactness argument on suitable measures to restore compactness.

We specifically highlight the work of Vaira \cite{vaira}, who investigated the system
\begin{equation}\label{eqvar}\begin{cases}
-\Delta u+u\pm\lambda K(x)\phi u=a(x)|u|^{k-2}u,&\qquad \text{in}\,\,\mathbb R^3\\
-\Delta\phi=K(x) u^2,&\qquad \text{in}\,\,\mathbb R^3
\end{cases}\end{equation}
for both sign cases, assuming $K$ and $a$ are nonnegative real functions satisfying suitable conditions. The author proves the existence of positive ground states, i.e. solutions with minimal energy, across the entire range $k \in (2, 6(=2^*))$ when the nonlocal term in \eqref{eqvar} carries a negative sign. Conversely, when the nonlocal term is positive, the competition between this term and the nonlinearity leads to substantially different scenarios depending on $k$. Consequently, the analysis in \cite{vaira} is restricted to $k \in (4,6)$. For further results, see also \cite{AmbrosettiMilan}, where \eqref{eqvar} is treated with the positive sign.

Owing to their robust physical foundations, systems analogous to \eqref{Lapla} and their generalizations have been subjects of extensive investigation over the past several decades. Schr\"odinger-Poisson systems featuring more general nonlinearities have been addressed in numerous studies \cite{amru, coc, avenia, articolopc_19, DSW23}. Furthermore, the case of the critical (conformal) dimension has been explored starting with the seminal work \cite{CW} and followed by \cite{GC, CLR}, employing a diverse range of analytical approaches.

Motivated by the physical arguments presented above, this paper focuses on the three-dimensional case, although the mathematical framework developed herein can be extended to any dimension $N \geq 2$ with suitable modifications. Specifically, we investigate generalizations of \eqref{Lapla} and \eqref{eqvar} involving operators with unbalanced growth, such as the $(p,q)$-Laplacian, which arises in models of nonlinear elasticity \cite{Zh86} and in the study of solitary waves for elementary particles \cite{gDeK, BDAFP00}. Furthermore, the $(2,4)$-Laplacian and its extensions appear as approximations of the Born–Infeld operator (see \cite{BInat, BI}).

It is worth noting that the presence of a nonlinear operator in the second equation of \eqref{maineq} introduces significant analytical challenges, beginning with the Poisson equation itself. Indeed, in the standard Laplacian case, the Lax-Milgram Theorem provides a representation formula for $\phi$. Conversely, when $m \neq 2$, an explicit expression for $\phi$ is generally unavailable. Nevertheless, the uniqueness of the nonnegative solution to the Poisson equation, as established in Proposition \ref{prop2.11} below, is sufficient to recover the key properties listed in Proposition \ref{331}.

Regarding the nonlinearity considered in the first equation of \eqref{maineq}, it can be viewed as a critical term perturbed by a subcritical one, both involving nontrivial weights. Nonlinearities of this type are of significant physical interest appearing, for instance, in the Yamabe problem or in the search for extremal functions for isoperimetric inequalities, and were first introduced by Brézis and Nirenberg \cite{BN} for the Laplacian case. Subsequently, Guedda and Véron \cite{GV89} extended several results from \cite{BN} to quasilinear equations driven by the $p$-Laplacian operator
$$-\Delta_p u=\lambda u^{p-1}+u^{p^*-1}\quad \text{in} \,\, \Omega,\qquad u>0 \quad \text{in} \,\, \Omega, \qquad u=0 \quad \text{on} \,\, \partial\Omega,$$
where $\Omega\subset \mathbb R^N$ is a bounded open set whose boundary is a $C^2$ submanifold of codimension $1$. In the same work, the authors established a general strong comparison principle for the $p$-Laplacian, extending the results of \cite{Tolk}, along with an extension of the Pohozaev identity \cite{PS}. On the other hand, the pure critical case in $\mathbb R^N$ was considered in  \cite{GV88}, where Guedda and Véron investigated positive radial solutions to $-\Delta_p u = u^{k-1}$ for $p < k \le p^*$, obtaining, among other things, that the best Sobolev constant for the embedding $W^{1,p}_0\hookrightarrow   L^{p^*}$ is achieved in $\mathbb R^N$, see \cite[Remark 3.1]{GV89}
and also providing a complete classification of isolated singularities in the subcritical case.

Moreover, in \cite{GV88T}, the same authors characterized the structure of the set of solutions for the eigenvalue problem with a potential in the one-dimensional case.   For Brézis-Nirenberg type problems involving the $p$-Laplacian in the entire Euclidean space, we refer to \cite{DH, SY}, while the non-homogeneous case of the $(p,q)$-Laplacian is addressed in \cite{BBF, BFccm} and in \cite{WS25} for normalized solutions.

In contrast, the study of Schr\"odinger-Poisson systems with critical nonlinearities appears to be a much more recent field of investigation. Key contributions include \cite{ZZ2009}, which addresses the case $p=q=2$ in $(\mathcal P_+)$, \cite{ZDS} for the case of fractional operators, and \cite{DSW2022} for the $p$-Laplacian operator. We also refer to \cite{LPNX} where multiplicity and concentration of normalized solutions of a double critical Schr\"odinger--Poisson system involving the fractional $p$-Laplacian is studied. More recently, the $(p,q)$-Laplacian was investigated in \cite{PLJ2024} looking for nodal solutions. Regarding $(\mathcal P_-)$, the nonlocal critical case was explored in \cite{TZ23}, while the zero-mass case for the $p$-Laplacian was studied in \cite{HS24}. Furthermore, for the $(p,q)$-Laplacian, we mention \cite{SHR}, where a fixed-point theory approach was employed to establish the existence of a positive solution to \eqref{maineq} (with $m=2, \vartheta=q$), highlighting the crucial role of the perturbation term in their argument.

A central difficulty in the variational study of classical Schr\"odinger-Poisson systems is ensuring the boundedness of Palais-Smale $(PS)$ sequences for the energy functional. Consequently, the majority of literature relies on specialized variational techniques that circumvent the need for a priori boundedness of all $(PS)$ sequences. These methods, introduced by Jeanjean \cite{J} and widely adopted in various contexts \cite{MZ, HIT}, typically involve constructing a specific bounded $(PS)$ sequence that satisfies an additional condition related to an auxiliary functional. Alternatively, Jeanjean and Le Coz \cite{JLC} developed another method based on an auxiliary truncated energy functional, which has been used in several settings \cite{LLS, AAP}.

One of the primary objectives of the present paper is to provide a comprehensive study of \eqref{maineq} for both sign cases. Unlike previous works described above, we employ a classical mountain pass approach. Although technically demanding, this allows us to identify the specific intervals where the standard mountain pass geometry is preserved and the boundedness, together with compactness, of the $(PS)$ sequences can be directly established.

While the Schr\"odinger-Poisson system is receiving increasing interest, the difficulties arising from the nonlinear nature of both equations in \eqref{maineq} mean that, to the best of our knowledge, this paper represents the first attempt to treat this type of problem variationally using a technique that differs from the standard ones. In this regard, we mention \cite{SHR}, where $\vartheta=q$ and $m=2$ in $(\mathcal P_-)$ with the addition of a perturbation term, and \cite{PLJ2024}, where nodal solutions to $(\mathcal P_+)$ were established for $\vartheta=q$ and $m=2$. Furthermore, in \cite{CLW}, quasilinear operators appear in both equations in $(\mathcal P_+)$, but the authors consider a subcritical nonlinearity. 
Notably, we investigate both signs of the nonlocal term, a feature that, outside of the Laplacian case \cite{vaira}, remains largely unexplored. Beyond the mathematical generality of the $(p,q)$-operator, our analysis is specifically designed to include the $q=2, p=4$ case, establishing a direct link to the physical foundations of Born-Infeld theory.

Our main result is the following
\begin{theorem}\label{th1}
Assume \eqref{K} and \eqref{W}. Then,
there exists $\, \beta^*>0$ such that
\begin{itemize}
\item  problem  $(\mathcal{P}_+)$ admits a nontrivial weak solution with positive energy for all $\beta>\beta^*$ and $\lambda>0$, provided 
that 
\begin{equation}\label{cond+}
\vartheta m'\le\tau<p^*\quad\mbox{and}\quad  \max\left\{1,\frac{q}{(m^*)'},\frac{p}{m'}\right\}<\vartheta<\frac{p^*}{m'};
\end{equation}
or 
\begin{equation}\label{cond++}
p<\tau<p^*\quad\mbox{and}\quad\max\left\{1,\dfrac{q}{(m^*)'}\right\}<\vartheta<\min\left\{\dfrac{p^*}{(m^*)'},\dfrac{p}{m'}\right\};
\end{equation}

\item problem  $(\mathcal{P}_-)$ admits a nontrivial weak solution with positive energy for all $\beta>\beta^*$ and $\lambda>0$ provided that  
\begin{equation}\label{cond-}
p<\tau<\min\{\vartheta m',p^*\}\quad\mbox{and}\quad\max\left\{1,\dfrac{q}{(m^*)'},\dfrac{p}{m'}\right\}<\vartheta<\dfrac{p^*}{(m^*)'}.
\end{equation}
\end{itemize}
\end{theorem}

The proof of Theorem \ref{th1} employs variational methods, primarily utilizing the Mountain Pass Theorem. We begin by establishing the appropriate functional framework and observing that \eqref{maineq} can be reduced to a single equation involving a nonlocal term. This reduction is typically achieved through representation theorems, as seen in the semilinear case where the Laplacian operator is involved, see \cite{vaira}. See also the papers by Mitidieri et al. \cite{DMP1, MP1, CDM} where the authors consider semilinear equations, inequalities or systems associated to general classes of differential operators, such as the Kohn Laplacian on the Heisenberg group, differential inequalities on Carnot groups or polyharmonic systems.

The presence of the quasilinear $m$-Laplacian in \eqref{maineq} complicates the process. Specifically, it precludes an explicit expression for $\phi_u$ for a fixed $u$, unlike the standard Laplacian case. Nevertheless, we report existence, uniqueness, and other essential properties that allow the argument to proceed.

As a second step, the boundedness of any Palais-Smale sequences at every energy level $c$ is established. This represents the first significant challenge, as the behavior of the nonlocal term varies depending on the range of the parameter $\vartheta$. Consequently, several cases must be analyzed separately, while also accounting for different ranges of $\tau$. Next, we verify that the energy functional satisfies the Mountain Pass geometry under more relaxed assumptions on the parameters. Given the intrinsic lack of compactness in \eqref{maineq} due to the presence of the critical Sobolev exponent and the ambient space $\mathbb R^3$, we must restore it via the Palais-Smale condition by employing concentration-compactness principles: the analysis of concentration at points follows Lions \cite{L84}, while concentration at infinity is addressed by Ben-Naoum et al. \cite{BNTW}. The final and most delicate step in ensuring the existence of a critical point is to demonstrate that the Mountain Pass level lies within the range where the compactness condition is satisfied. It should be noted that, while the general conditions on $\tau$ and $\vartheta$ are stated in Theorem \ref{th1}, each preliminary lemma provides more refined requirements. In some cases, these conditions prove to be less restrictive than those in \eqref{cond+}, \eqref{cond++}, and \eqref{cond-}.

Although this strategy follows a classical variational framework, its implementation introduces numerous technical assumptions on the parameters $\tau$ and $\vartheta$ of the subcritical and nonlocal terms, respectively.

The paper is organized as follows. Section \ref{prel} provides an overview of standard concepts and results, such as properties for the Poisson equation, concentration compactness principles and the Mountain Pass Theorem. In Section \ref{var} we introduce the functional framework giving the foundation to implement the variational approach. Finally, Section \ref{proof}, starting from the properties of the energy functional and its behavior, is devoted to the rigorous proofs of the main result of the paper, Theorem \ref{th1}.

\section{Introductory tools}\label{prel}

In this section, we recall the main definitions and properties of Sobolev spaces together with classical results. Then, we give a brief outline of the regularity properties of the functionals associated with our main problems.

We begin by reviewing some notation. Since this paper focuses on the three-dimensional case, we shall maintain this assumption throughout the section, although results such as Sobolev embeddings or concentration compactness principles hold in higher dimensions.

We indicate with $B_r(x)$ the $\mathbb R^3$-ball of center $x\in\mathbb R^3$ and radius $r>0$, omitting $x$ when it is the origin.
Let $(X, \|\cdot\|)$ be a  Banach space, we indicate with $X'$ its dual, while $\langle \cdot, \cdot \rangle$ stands for the duality brackets for the pair $(X,X')$. Given two Banach spaces $X,Y$, the continuous embedding of $X$ into $Y$ is indicated by $X \hookrightarrow Y$; if the embedding is compact, we write $X \hookrightarrow \hookrightarrow  Y$. 

A sequence $(u_n)_n\in X$ strongly converges to $u$ when $u_n \to u$ in $X$. If the convergence is in weak sense then $u_n \rightharpoonup u$ in $X$.

Given any measurable set $\Omega \subset \mathbb R^3$ and $r \in [1,+\infty]$, $L^r(\Omega)$ stands for the standard Lebesgue space, whose norm will be indicated with $ \| \cdot \|_{L^r(\Omega)}$ or simply $\| \cdot \|_r$ when $\Omega=\mathbb R^3$. Moreover, we denote $r':=r/(r-1)$ as the conjugated index of $r$.

We will also make use of the Sobolev space $W^{1,r}(\mathbb R^3)$, which is the closure of smooth functions with compact support, i.e. $C^{\infty}_c(\mathbb R^3)$, with respect to the norm $$
\|u\|_{1,r}=\|u\|_r+\|\nabla u\|_r.
$$

The continuous embeddings of Sobolev spaces into Lebesgue spaces are clearly described in Sobolev-Gagliardo-Nirenberg Theorem stated below.

\begin{theorem}[\cite{Brezisaf}]\label{sgn}
Let $1 \le p <3$. Then $$
W^{1,p}(\mathbb R^3) \hookrightarrow L^{p^*}(\mathbb R^3), 
$$
where $p^*=\frac{3p}{3-p}$ is the critical Sobolev exponent. In particular,
$$W^{1,p}(\mathbb R^3) \hookrightarrow L^{s}(\mathbb R^3), \,\, \forall s \in [p,p^*].
$$
\end{theorem}

Since we are working with the $(p,q)$-Laplacian operator in \eqref{maineq}, then we will define $(X, \|\cdot\|)$ as
$$X=W^{1,p}(\mathbb R^3)\cap  W^{1,q}(\mathbb R^3),\qquad\qquad \|u\|:=\|u\|_{1,p}+\|u\|_{1,q}.$$
From Theorem \ref{sgn}, we have the following embedding
\begin{equation}\label{embX}
X\hookrightarrow L^s(\mathbb R^3), \,\, \forall s \in [q,p^*].
\end{equation}

In some cases we will need to use a larger space, so called homogeneous Sobolev space or Beppo Levi space, given by $D^{1,m}(\mathbb R^3)$ where $1<m<3$, defined 
as the closure of $C^\infty_c(\mathbb R^3)$ with respect to the norm
$$\|u\|_{D^{1,m}(\mathbb R^3)}:=\|\nabla u\|_m.$$
If $1<m<3$, then $D^{1,m}(\mathbb R^3)$ is reflexive, separable Banach space and we indicate with $D^{-1,m'}(\mathbb R^3)$ the dual of $D^{1,m}(\mathbb R^3)$. The continuous embedding of the $D^{1,m}(\mathbb R^3)$ is known as Sobolev Theorem, here reported for completeness.
\begin{theorem}[\cite{analisimoderna}] \label{immD12}  Let $1<m<3$. Then $$D^{1,m}(\mathbb R^3) \hookrightarrow L^{m^*}(\mathbb R^3),$$
and the best constant in the Sobolev inequality will be indicated with $S^{-1/m}$, where
\begin{equation}\label{S}
S= \inf_{u\in D^{1,m}(\mathbb R^3,\mathbb R)\setminus\{0\}} \frac{\|\nabla u\|_{{m}}^m}{\|u\|_{m^*}^m}
\end{equation}
\end{theorem}
According to which, one has $$D^{1,m}(\mathbb R^3) = \left\{u\in L^{m^*}(\mathbb R^3): \, |\nabla u|\in L^m(\mathbb R^3)\right\}.$$ 

Now we will focus our attention on the Poisson equation in \eqref{maineq}, starting with an existence and uniqueness result and ending with properties of such solution, following \cite{DSW23}.

 \begin{proposition}\label{prop2.11}
Let $1<q<p<3$, $u\in X$ and $\vartheta$ satisfies \eqref{defthetamain}.
Then there is a unique nonnegative solution $\phi_u\in D^{1,m}(\mathbb R^3)$ to 
\begin{equation}\label{probphi}
-\Delta_m \phi =|u|^\vartheta\quad \text{in}\,\, \mathbb R^3. 
\end{equation}
\end{proposition}

We do not report its proof, as the one of the Proposition below, since they come straightforwardly from Propositions 2.1 and 2.2 in \cite{DSW23} by using \eqref{embX}. 

It is hard to give an explicit expression of $\phi_u$. We can prove the following properties of $\phi_u$ via the uniqueness of the solution of \eqref{probphi}.

\begin{proposition}\label{331}
For $u \in X$ and $\vartheta$ satisfy \eqref{defthetamain}, the solution $\phi_u$ of \eqref{probphi} given by Proposition \ref{prop2.11} has the following properties.
\begin{itemize}
    \item[(i)] It holds $\phi_u \geq 0$ and
    $$
    \int_{\mathbb R^3} \left( \frac{1}{m} |\nabla \phi_u|^m - |u|^\vartheta \phi_u \right) {\rm d} x = 
    \min_{\phi \in D^{1,m}(\mathbb R^3)} \int_{\mathbb R^3} \left( \frac{1}{m} |\nabla \phi|^m - |u|^\vartheta \phi \right) {\rm d} x.
    $$

    \item[(ii)] For $t > 0$, $\phi_{tu} = t^{\frac{\vartheta}{m-1}} \phi_u$ and 
    $\phi_{u_t}(x) = t^{\frac{k\vartheta - m}{m}} \phi_u(tx)$, where $u_t(x) = t^k u(tx)$. 
    Moreover $\phi_{u(\cdot + y)} = \phi_u(\cdot + y)$ for any $y \in \mathbb R^3$.

    \item[(iii)] It holds
    $$
    \|\phi_u\|_{D^{1,m}} \leq C \|u\|^{\frac{\vartheta}{m-1}}, \quad \text{where $C$ does not depend on } u.
    $$

    \item[(iv)] If $u_n  \rightharpoonup u$ in $X$, then $\phi_{u_n}  \rightharpoonup \phi_u$ in $D^{1,m}(\mathbb R^3)$ and
    $$
    \int_{\mathbb R^3} \phi_{u_n} |u_n|^{\vartheta-2} u_n \varphi {\rm d} x 
    \to 
    \int_{\mathbb R^3} \phi_u |u|^{\vartheta-2} u \varphi {\rm d} x, 
    \quad \forall \varphi \in X.
    $$

    \item[(v)] If $u_n \to u$ in $X$, then $\phi_{u_n} \to \phi_u$ in $D^{1,m}(\mathbb R^3)$.
\end{itemize}

\end{proposition}

To address the lack of compactness proper of problems with a critical nonlinearity in the entire Euclidean space $\mathbb R^3$, we introduce the following two lemmas. These tools are specifically designed to handle concentration of compactness at points and at infinity, respectively.
\begin{lemma}[{\cite[Lemma I.1]{L3}}]
\label{lions}
Let $1\leq p<3$. Suppose $(u_n) \subseteq D^{1,p}(\mathbb R^3)$ to be such that $u_n \rightharpoonup u$ in $D^{1,p}(\mathbb R^3)$, and both $|\nabla u_n|^p \rightharpoonup \mu$, $|u_n|^{p^*} \stackrel{*}{\rightharpoonup} \nu$ in the sense of measures, for some $u\in D^{1,p}(\mathbb R^3)$ and $\mu,\nu$ bounded non-negative measures on $\mathbb R^3$. Then there exist some at most countable set $J$, a family of distinct points $(x_j)_{j\in J}\subseteq \mathbb R^3$, and two families of numbers $(\nu_j)_{j\in J}, (\mu_j)_{j\in J}\subseteq (0,+\infty)$ fulfilling
\begin{equation}
\label{ineqmeasures}
\nu=|u|^{p^*}+\sum_{j\in J} \nu_j \delta_{x_j},\quad
\mu\geq |\nabla u|^p+\sum_{j\in J} \mu_j \delta_{x_j},\quad
S\nu_j^{p/p^*} \leq \mu_j \quad  j\in J,
\end{equation}
where $S$ is the best constant in the Sobolev inequality defined in \eqref{S} with $m=p$.
\end{lemma}

Note that Lemma \ref{lions} requires the tight convergence of the measures involving the critical Sobolev exponent, but the proof of this condition reveals to be rather difficult and technical. Thus, Ben-Naoum et al. established a version of the Lemma \ref{lions} known as {\it escape to infinity principle}, where the concentration at infinity is enclosed in the parameters $\nu_\infty$ and $\mu_\infty$.

\begin{lemma}[{\cite[Lemma 3.3]{BNTW}}]
\label{bennaoum}
Let $1\leq p<3$. Suppose that $(u_n) \subseteq D^{1,p}(\mathbb R^3)$ is bounded and define
\begin{equation}\label{munuinf}\begin{aligned}
&\nu_\infty := \lim_{R\to+\infty} \limsup_{n\to\infty} \int_{B_R^c} |u_n|^{p^*} {\rm d} x,\\& \mu_\infty := \lim_{R\to+\infty} \limsup_{n\to\infty} \int_{B_R^c} |\nabla u_n|^p {\rm d} x.\end{aligned}\end{equation}
Then, it holds $S\nu_\infty^{p/p^*} \leq \mu_\infty$ and
\begin{equation}\label{nunuinfty}\begin{aligned}
&\limsup_{n\to\infty} \int_{\mathbb R^3} |u_n|^{p^*} {\rm d} x = \int_{\mathbb R^3} \, {\rm d}\nu + \nu_\infty, 
\\&\limsup_{n\to\infty} \int_{\mathbb R^3} |\nabla u_n|^p {\rm d} x = \int_{\mathbb R^3} \, {\rm d}\mu + \mu_\infty, 
\end{aligned}\end{equation}
where $\nu,\mu$ are as in Lemma \ref{lions}.
\end{lemma}

We will use Lemmas \ref{lions}--\ref{bennaoum} to avoid concentration both at points, i.e. $\nu_j=\mu_j=0$ for all $j\in J$, and at
infinity, i.e. $\nu_\infty=\mu_\infty=0$.

\bigskip

We end the present section by introducing the following version of the Mountain Pass Theorem.

\begin{theorem}[{\cite{AR}}]\label{mptheorem}
Let $(V,\| .\|_V)$ be a Banach space and consider $F\in C^1(V)$. We assume that
\par (i) $F(0)=0$,
\par (ii) There exist $\alpha,R>0$ such that $F(u)\geq\alpha$ for all $u\in V$, with $\|u\|_V=R$,
\par (iii) There exists $v_0\in V$ such that $\limsup_{t\to \infty}F(tv_0)<0$.
\par\noindent
Let $t_0>0$ be such that $\|t_0v_0\|_V>R$ and $F(t_0v_0)<0$
and let
$$c:=\inf_{\gamma\in\Gamma}\,\sup_{t\in [0,1]}F(\gamma(t)),$$
where
$$\Gamma:=\{\gamma\in C^0([0,1],V) :  \gamma(0)=0\hbox{ and }\gamma(1)=t_0v_0\}.$$
Then, there exists a Palais-Smale sequence at level $c$, that is
 a sequence $(u_n)_n\subset V$ such that
$$\lim_{n\to \infty}F(u_n)=c\quad\hbox{ and }\quad\lim_{n\to \infty}F'(u_n)= 0\quad\hbox{strongly in }V'.$$
\end{theorem}

\section{The variational approach}\label{var}

Now we establish the variational framework of \eqref{maineq}.

By  Proposition \ref{prop2.11}, for any $u\in X$, there is a unique nonnegative solution $\phi_u\in D^{1,m}(\mathbb R^3)$ to \eqref{probphi}. Thus the map $\mathcal{P}hi:X \to D^{1,m}(\mathbb R^3)$ defined as $\mathcal{P}hi(u)=\phi_u,$
is well-defined from the above argument. 
Now we define the functional
$\mathcal{F}: X \times D^{1,m}(\mathbb R^3) \to \mathbb R$ as
$$\begin{aligned}
\mathcal{F}(u, \phi):= \frac{1}{p}\|u\|_{1,p}^p+\frac{1}{q}\|u\|_{1,q}^q&\pm \frac{\lambda}{\vartheta}\int_{\mathbb R^3} \phi |u|^{\vartheta}{\rm d} x\mp\frac{\lambda}{m \vartheta}\int_{\mathbb R^3} |\nabla \phi|^{m}{\rm d} x\\
&-\frac{\beta}{\tau}\int_{\mathbb R^3} \mathcal{W}(x)|u|^{\tau}{\rm d} x-\frac{1}{p^*}\int_{\mathbb R^3} \mathcal{K}(x)|u|^{p^*}{\rm d} x,
\end{aligned}$$
for any $u \in X$ and $\phi \in D^{1,m}(\mathbb R^3)$. By standard arguments, the functional $\mathcal{F}$ is of class $C^1$ and its critical points are weak solutions of \eqref{maineq}. In particular, its partial derivatives are
$$\begin{aligned}
\partial_u\mathcal{F}(u,\phi)[v]=&\int_{\mathbb R^3} |\nabla u|^{p-2}\nabla u \nabla v {\rm d} x+ \int_{\mathbb R^3} |u|^{p-2}uv {\rm d} x+\int_{\mathbb R^3} |\nabla u|^{q-2}\nabla u \nabla v {\rm d} x\\ &+ \int_{\mathbb R^3} |u|^{q-2}uv {\rm d} x \pm \lambda\int_{\mathbb R^3}  \phi |u|^{\vartheta-2}uv {\rm d} x\\ &- \beta\int_{\mathbb R^3} \mathcal{W}(x) |u|^{\tau -2}u v {\rm d} x- \int_{\mathbb R^3} \mathcal{K}(x) |u|^{p^* -2}u v {\rm d} x,
\end{aligned}$$
$$\partial_{\phi}\mathcal{F}(u,\phi)[\xi]= \pm\frac{\lambda}{\vartheta}\int_{\mathbb R^3} |u|^{\vartheta} \xi {\rm d} x\mp\frac{\lambda}{\vartheta}\int_{\mathbb R^3} |\nabla\phi|^{m-2}\nabla\phi \nabla \xi {\rm d} x.
$$ Let us define $$
\mathcal{J}(u)=\mathcal{F}(u,\mathcal{P}hi(u)),
$$
then since $\phi_u$ solves \eqref{probphi}, then $\mathcal{J}$ takes the form
\begin{equation*}\begin{aligned}
\mathcal{J}(u)=\frac{1}{p}\|u\|_{1,p}^p&+\frac{1}{q}\|u\|_{1,q}^q\pm \frac{\lambda}{ \vartheta m'}\int_{\mathbb R^3} \phi_u |u|^{\vartheta}{\rm d} x\\&-\frac{\beta}{\tau}\int_{\mathbb R^3} \mathcal{W}(x)|u|^{\tau}{\rm d} x-\frac{1}{p^*}\int_{\mathbb R^3} \mathcal{K}(x)|u|^{p^*}{\rm d} x.
\end{aligned}\end{equation*}

Although our objective is to study problem \eqref{maineq} through a unified approach, it will become necessary at certain stages to distinguish between the two cases. For the sake of clarity, we define the functionals corresponding to the plus and minus signs as $J_+$ and $J_-$, respectively.

\begin{equation}\label{defJ+}\begin{aligned}
\mathcal{J}_+(u)=\frac{1}{p}\|u\|_{1,p}^p&+\frac{1}{q}\|u\|_{1,q}^q+ \frac{\lambda}{ \vartheta m'}\int_{\mathbb R^3} \phi_u |u|^{\vartheta}{\rm d} x\\&-\frac{\beta}{\tau}\int_{\mathbb R^3} \mathcal{W}(x)|u|^{\tau}{\rm d} x-\frac{1}{p^*}\int_{\mathbb R^3} \mathcal{K}(x)|u|^{p^*}{\rm d} x,
\end{aligned}\end{equation}
\begin{equation}\label{defJ-}\begin{aligned}
\mathcal{J}_-(u)=\frac{1}{p}\|u\|_{1,p}^p&+\frac{1}{q}\|u\|_{1,q}^q- \frac{\lambda}{ \vartheta m'}\int_{\mathbb R^3} \phi_u |u|^{\vartheta}{\rm d} x\\&-\frac{\beta}{\tau}\int_{\mathbb R^3} \mathcal{W}(x)|u|^{\tau}{\rm d} x-\frac{1}{p^*}\int_{\mathbb R^3} \mathcal{K}(x)|u|^{p^*}{\rm d} x.
\end{aligned}\end{equation}

Clearly, $\mathcal{J} \in C^1(X, \mathbb R)$ and for any $u \in X$,
$$
\mathcal{J}'(u)=\partial_u(\mathcal{F}(u,\mathcal{P}hi(u))+\partial_{\phi}(\mathcal{F}(u,\mathcal{P}hi(u))\mathcal{P}hi'(u).
$$
However, recalling that
$$
\partial_{\phi}\mathcal{F}(u,\mathcal{P}hi(u))=0,
$$
we can write
$$
\begin{aligned}
\mathcal{J}'(u)[v]=  &\int_{\mathbb R^3} |\nabla u|^{p-2}\nabla u \nabla v {\rm d} x+ \int_{\mathbb R^3} |u|^{p-2}uv {\rm d} x+\int_{\mathbb R^3} |\nabla u|^{q-2}\nabla u \nabla v {\rm d} x\\ &+ \int_{\mathbb R^3} |u|^{q-2}uv {\rm d} x \pm \lambda\int_{\mathbb R^3}  \phi_u |u|^{\vartheta -2}uv {\rm d} x  \\ &- \beta\int_{\mathbb R^3} \mathcal{W}(x) |u|^{\tau-2}u v {\rm d} x- \int_{\mathbb R^3} \mathcal{K}(x) |u|^{p^*-2}u v {\rm d} x,
\end{aligned}$$
for any $v \in X$, see Proposition 2.3 in \cite{DSW23}.
Therefore to look for weak solutions of \eqref{maineq}, it suffices to look for critical points of the functional $\mathcal{J}$.

\section{Proof of Theorem \ref{th1}}\label{proof}

With the variational framework now established, we can delve into the core of the proof of Theorem~\ref{th1}. We begin with taking Palais-Smale sequences associated to the functional $\mathcal{J}$ and first proving their boundedness. However, due to the presence of the nonlocal term, this is possible only in some cases.

\begin{lemma}\label{lembound}
Assume \eqref{defthetamain}, \eqref{K}, \eqref{W}  and $p< \tau <p^*$.
\begin{itemize}
\item Let $(u_n)_n$ be a $(PS)_c$ sequence for $\mathcal{J}_+$ in $X$, with $c\in \mathbb R$ under the assumptions
\begin{equation}\label{bound:PS:J+}
\max\{p,\vartheta m'\}<\tau<p^*\qquad \max\left\{1,\frac{q}{(m^*)'}\right\}<\vartheta<\frac{p^*}{m'} 
\end{equation}
    or
\begin{equation}\label{bound:PS:J+2}
\tau=\vartheta m' \qquad \max\left\{1,\frac{q}{(m^*)'},\frac{p}{m'}\right\}<\vartheta<\frac{p^*}{m'}
\end{equation}
Then, $(u_n)_n$ is bounded in $X$. 
\item Let $(u_n)_n$ be a $(PS)_c$ sequence for $\mathcal{J}_-$ in $X$, with $c\in \mathbb R$ under the assumptions \eqref{cond-}
    or \eqref{bound:PS:J+2}.
Then, $(u_n)_n$ is bounded in $X$.
\end{itemize}
\end{lemma}

\begin{proof} 
Take any $(u_n)_n$ $(PS)_c$ sequence for $\mathcal{J}$ in $X$, with $c\in \mathbb R$. By definition
$$\mathcal{J}(u_n)\to \, c \quad\mbox{and}\quad  \mathcal{J}'(u_n)\to 0 \,\,\mbox{in } X',$$
as $n\to\infty$. Thus,
\begin{equation}\label{PSbound}\begin{aligned}
c+&o(1)=\mathcal{J}(u_n)-\frac{1}{\tau}\mathcal{J}'(u_n)u_n\\&=\left(\frac{1}{p}-\frac{1}{\tau}\right)\|u_n\|_{1,p}^p+\left(\frac{1}{q}-\frac{1}{\tau}\right)\|u_n\|_{1,q}^q\\ &\pm \frac{\lambda}{ \vartheta}\left(\frac{1}{m'}-\frac{\vartheta}{\tau}\right)\int_{\mathbb R^3} \phi_{u_n} |u_n|^{\vartheta}{\rm d} x-\left(\frac{1}{p^*}-\frac{1}{\tau}\right)\int_{\mathbb R^3} \mathcal{K}(x)|u_n|^{p^*}{\rm d} x\\&\ge \left(\frac{1}{p}-\frac{1}{\tau}\right)(\|u_n\|_{1,p}^p+\|u_n\|_{1,q}^q)\pm \frac{\lambda}{ \vartheta}\left(\frac{1}{m'}-\frac{\vartheta}{\tau}\right)\int_{\mathbb R^3} \phi_{u_n} |u_n|^{\vartheta}{\rm d} x, 
\end{aligned}\end{equation}
where we have used $\tau<p^*$, $q\le p$ and \eqref{K}.

Both for $\mathcal{J}_+$ or $\mathcal{J}_-$, by employing the relative assumptions \eqref{bound:PS:J+}-\eqref{bound:PS:J+2} and \eqref{cond-}-\eqref{bound:PS:J+2}, we obtain the nonnegativity of the nonlocal term. 

So that \eqref{PSbound} gives
 $$\begin{aligned}c+o(1)\ge&\left(\frac{1}{p}-\frac{1}{\tau}\right)(\|u_n\|_{1,p}^p+\|u_n\|_{1,q}^q).
 \end{aligned}$$
 Letting $\|u_n\|=\|u_n\|_{1,p}+\|u_n\|_{1,q}\to\infty$, the conclusion follows easily by contradiction.
\end{proof}

\begin{remark}
Note that the assumptions on $\tau$ and $\vartheta$ in Lemma \ref{lembound} are necessary to ensure the positivity of the nonlocal term's coefficients for both $\mathcal{J}_+$ and $\mathcal{J}_-$. This allows the term to be bounded from below by zero. Note that adopting a negative coefficient would make it impossible to compare the operator's norm with the $L^{p^*}$ norm of the critical term. This difficulty persists even when using a weighting factor different from $1/\tau$ and estimating the nonlocal term by \eqref{disphi}.
\end{remark}

Next, we verify the Mountain Pass geometry and study the Palais-Smale sequence given by Theorem \ref{mptheorem}, recovering the compactness necessary to ensure the existence of a solution.

\begin{lemma}\label{pl106}
Assume \eqref{W}, \eqref{K}, \eqref{defthetamain}, $p\le \tau <p^*$.
Then, properties
\begin{itemize}
\item[(i)] there exists $\rho>0$ and $\delta>0$ so that $\mathcal{J}(u)\ge \delta$ for any $u \in X$ with $\|u\|=\rho$,
\item[(ii)] there exists $v \in X$ such that $\|v\|>\rho$ and $\mathcal{J}(v)<0$
\end{itemize}
 hold for $\mathcal{J}_+$ if
 \begin{equation*}
\max\left\{1,\frac{q}{(m^*)'}\right\}<\vartheta<\frac{p^*}{m'}.
\end{equation*}
While they hold for $\mathcal{J}_-$ if
\begin{equation}\label{theta_new}
\max\left\{1,\frac{q}{(m^*)'}, \frac{p}{m'}\right\}<\vartheta<\frac{p^*}{(m^*)'}.
\end{equation}
\end{lemma}

\begin{proof}
(i) We divide the proof in two cases:
\begin{itemize}
\item[Case $\mathcal{J}_+$:] Using H\"older inequality we infer that

$$\begin{aligned}\mathcal{J}_+(u)
\ge \frac{1}{p}(\|u\|_{1,p}^p+\|u\|_{1,q}^q)-\frac{\beta}{\tau}\|\mathcal{W}\|_{\eta}\|u\|_{p^*}^\tau-\frac{1}{p^*}\|\mathcal{K}\|_{\infty}\|u\|_{p^*}^{p^*}
\end{aligned}$$

Now, taking $\|u\|\le 1$, then 
$\|u\|_{1,p},\|u\|_{1,q}\le 1$ so that, being $q\le p$ we have
$\|u\|_{1,q}^q\ge \|u\|_{1,q}^p$ and by the standard inequality $(a+b)^r\le \max\{1,2^{r-1}\}(a^r+b^r)$ for all $a,b,r>0$, also by Theorem \ref{sgn} we reach
$$\begin{aligned}\mathcal{J}_+(u)&\gtrsim \|u\|^p-\frac{\beta}{\tau}\|\mathcal{W}\|_{\eta}\|u\|^\tau-\frac{1}{p^*}\|\mathcal{K}\|_{\infty}\|u\|^{p^*}
\\
&= \|u\|^p\biggl(1-\frac{\beta}{\tau}\|\mathcal{W}\|_{\eta}\|u\|^{\tau-p}\frac{1}{p^*}\|\mathcal{K}\|_{\infty}\|u\|^{p^*-p}\bigr). 
\end{aligned}$$

\item[Case $\mathcal{J}_-$:] 

First note that, by using Holder inequality, Proposition \ref{331}, Theorem \ref{immD12}, \eqref{embX} by \eqref{defthetamain}, we get
\begin{equation}\label{disphi}\begin{aligned}
\int_{\mathbb R^3}& \phi_{u_n} |u_n|^{\vartheta} {\rm d} x\le \biggl(\int_{\mathbb R^3} |\phi_{u_n}|^{m^*} {\rm d} x \biggr)^{\frac{1}{m^*}}\biggl ( \int_{\mathbb R^3} |u_n|^{(m^*)'\vartheta}{\rm d} x\biggr) ^{\frac{1}{(m^*)'}}  \\
&\le C \|\phi_{u_n}\|_{D^{1,m}} \|u_n\|^\vartheta \le C \|u_n\|^{\frac{\vartheta}{m-1}} \|u_n\|^\vartheta = C\|u_n\|^{\vartheta m'}.
\end{aligned}\end{equation}
Thus, applying Theorem \ref{sgn} and \eqref{disphi}, we infer that
$$\mathcal{J}_-(u)\gtrsim \frac{1}{p}(\|u\|_{1,p}^p+\|u\|_{1,q}^q)- \frac{\lambda}{ \vartheta}\frac{1}{m'}\|u\|^{\vartheta m'}-\frac{\beta}{\tau}\|\mathcal{W}\|_{\eta}\|u\|^\tau-\frac{1}{p^*}\|\mathcal{K}\|_{\infty}\|u\|^{p^*}$$

Now, taking $\|u\|\le 1$, as above, then 
$\|u\|_{1,q}^q\ge \|u\|_{1,q}^p$ and by the standard inequality $(a+b)^r\le \max\{1,2^{r-1}\}(a^r+b^r)$ for all $a,b,r>0$, we reach
$$\begin{aligned}\mathcal{J}_-(u)&\gtrsim  \|u\|^p- \frac{\lambda}{ \vartheta m'}\|u\|^{\vartheta m'}-\frac{\beta}{\tau}\|\mathcal{W}\|_{\eta}\|u\|^\tau-\frac{1}{p^*}\|\mathcal{K}\|_{\infty}\|u\|^{p^*}
\\
&= \|u\|^p\biggl(1- \frac{\lambda}{ \vartheta m'}\|u\|^{\vartheta m'-p}-\frac{\beta}{\tau}\|\mathcal{W}\|_{\eta}\|u\|^{\tau-p}\frac{1}{p^*}\|\mathcal{K}\|_{\infty}\|u\|^{p^*-p}\bigr). 
\end{aligned}$$

\end{itemize}

Since $\tau>p$ and $\vartheta >p/m'$ in \eqref{theta_new}, then in both cases above there  exists a sufficiently small norm $\|u\|:=\rho (<1)$ and $\delta >0$ such that $\mathcal{J}(u)\geq \delta$ for every $u$ satisfying $\|u\|=\rho$.

(ii) For any $u\in X\setminus \{0\}$, by Proposition \ref{331}, we get
$$\begin{aligned}
\mathcal{J}(tu)=\frac{t^p}{p}\|u\|_{1,p}^p&+\frac{t^q}{q}\|u\|_{1,q}^q \pm t^{m'\vartheta} \frac{\lambda}{m' \vartheta }\int_{\mathbb R^3} \phi_{u} |u|^{\vartheta} {\rm d} x \\&-\frac{t^{\tau}}{\tau} \int_{\mathbb R^3} \mathcal{W}|u|^{\tau} {\rm d} x-\frac{t^{p^*}}{p^*}\int_{\mathbb R^3} \mathcal{K} |u|^{p^*}.
\end{aligned}$$

Now, recalling that $\mathcal{K}$ is nontrivial and satisfy \eqref{K}, we have two different situations:
\begin{itemize}
\item[Case $\mathcal{J}_+$:] Since $q<p\le\tau<p^*$ and assuming $\vartheta<p^*/m'$, then $\mathcal{J}_+(tu)\to-\infty$ as $t\to\infty$.
\item[Case $\mathcal{J}_-$:] Since $q<p\le\tau<p^*$, then $\mathcal{J}_-(tu)\to-\infty$ as $t\to\infty$.
\end{itemize}
Thus, in both cases, there exists $t_u >0$ large such that $\mathcal{J}(t_u u)<0$. Consequently, (ii) holds with $v=t_u u$.
\end{proof}

Consider for all $u\in X\setminus\{0\}$
$$\Gamma_u:=\{\gamma\in C^0([0,1],X)\,:\, \gamma(0)=0\hbox{ and }\gamma(1)=v\},$$
where $v$ is given in (ii)-Lemma \ref{pl106} with $\mathcal{J}(v)<0$ and $\|v\|>\rho$.  
Then for all $\gamma\in \Gamma_u$, it holds
$$\|\gamma(0)\|=0,\qquad
\|\gamma(1)\|>\rho,\qquad
\gamma \text{ continuous}.$$
Thus, then exists $\bar t\in (0,1)$ such that $\|\gamma(\bar t)\|=\rho$
for all $\gamma\in\Gamma_u$ implying 
$$\max_{t\in[0,1]}{\mathcal{J}(\gamma(t))}\geq \delta >0,$$
by (i)-Lemma \ref{pl106}.
Thus,
$$c_u:=\inf_{\gamma\in\Gamma_u}\,\sup_{t\in [0,1]}\mathcal{J}(\gamma(t))>0.$$
Then the hypotheses of Theorem \ref{mptheorem} are satisfied, yielding the existence of a Palais-Smale sequence at level $c_u$. In order to avoid ambiguity, when necessary, we consider
\begin{equation}\label{cu+-}\begin{aligned}
&c_u^+:=\inf_{\gamma\in\Gamma_u}\,\sup_{t\in [0,1]}\mathcal{J}_+(\gamma(t)),\qquad c_u^-:=\inf_{\gamma\in\Gamma_u}\,\sup_{t\in [0,1]}\mathcal{J}_-(\gamma(t))
\end{aligned}\end{equation}

In what follows, we establish the compactness properties for the functional $\mathcal{J}$. This analysis requires a case-by-case approach based on the sign of the nonlocal term. Specifically, we treat $\mathcal{J}_+$ and $\mathcal{J}_-$ as defined in \eqref{defJ+} and \eqref{defJ-}.

\begin{lemma}\label{pscomp}
Assume \eqref{W}, \eqref{K}, \eqref{defthetamain} and $p< \tau <p^*$. Define 
\begin{equation}\label{Csegnato}
c_{PS}:=\frac{S^{3/p}}{3\|\mathcal{K}\|_\infty^{3/p^*}}
\end{equation}
\begin{equation}\label{Csegnato'}
c_{PS}':=\left(\frac{1}{\vartheta m'}-\frac{1}{p^*}\right)\frac{S^{3/p}}{3\|\mathcal{K}\|_\infty^{3/p^*}}\quad \text{with}\quad \vartheta m'<p^*
\end{equation}
and note that $c_{PS}'<c_{PS}$.

Then, 
$\mathcal{J}_+$ satisfies the $(PS)_c$ condition for every 
\begin{itemize}
\item $c<c_{PS}'$ if \eqref{cond+} holds
\item $c<c_{PS}$ if  \eqref{cond++} holds
\end{itemize}
and $\mathcal{J}_-$ satisfies the $(PS)_c$ condition for every 
\begin{itemize}
\item $c<c_{PS}$ if \eqref{cond-} holds
\end{itemize}
\end{lemma}

\begin{remark}
Note that while for $\mathcal{J}_-$ the conditions in Lemma \ref{pscomp} and Lemma \ref{PSbound} coincide, conditions \eqref{cond+} and \eqref{cond++} are stronger than \eqref{bound:PS:J+} and \eqref{bound:PS:J+2}. Consequently, Lemma \ref{lembound} applies, ensuring the boundedness of $(PS)$ sequences.
\end{remark}

\begin{proof}[Proof of Lemma \ref{pscomp}]
Let $(u_n)_n$ be a $(PS)_c$ sequence in $X$, so that, by Lemma \ref{lembound}, then $(u_n)_n$ is bounded in $X$, indeed assumptions on $\tau,\theta$ above are stronger than \eqref{bound:PS:J+}, \eqref{bound:PS:J+2}, \eqref{cond-}. 
Since $X$ is a reflexive Banach space, then, by Banach-Alaoglu Theorem,
 there exists $u\in X$  such that, up to subsequences, we get
$u_{n}\rightharpoonup u$ in $X$.
 Since $\nabla u_{n}\rightharpoonup \nabla u$ in $L^{p}(\mathbb R^3)\cap L^{q}(\mathbb R^3)$, the sequence of measures
 $(|\nabla u_n|^p {\rm d} x+|\nabla u_n|^q {\rm d} x)_n$ is bounded and
$|\nabla u_{n}|^{p}{\rm d} x+|\nabla u_n|^q {\rm d} x\rightharpoonup \mu$. Analogously,
$|u_n|^{p^*}{\rm d} x\rightharpoonup \nu,$
where $\mu, \nu$ are bounded nonnegative measures on $\mathbb R^3$.

By Proposition \ref{lions} and Lemma \ref{bennaoum}, there exist at most countable set $J$, a family $(x_j)_{j\in J}$ of distinct points in
$\mathbb R^3$ and two families $(\nu_j)_{j\in J}, \,(\mu_j)_{j\in J}\in ]0,\infty[$ such that
\eqref{ineqmeasures}, \eqref{nunuinfty} hold, with $\nu_\infty$, $\mu_\infty$ defined in \eqref{munuinf}, satisfying
\begin{equation}\label{6.22}
S\nu_{j}^{p/p^{*}}\le\mu_{j}, \qquad S\nu_{\infty}^{p/p^*}\le \mu_\infty.
\end{equation}

Take a standard cut-off function $\psi\in C_{c}^{\infty}(\mathbb R^3)$, such that $0\le\psi\le1$ in $\mathbb R^3$,
$\psi=0$ for $|x|>1$, $\psi=1$ for $|x|\le 1/2$. For each index $j\in J$ and each $0<\varepsilon<1$, define
$$\psi_{\varepsilon}(x)=\psi\left(\frac{x-x_{j}}{\varepsilon}\right).$$

Since $\mathcal{J}'(u_{n})\psi\to0$ being $(u_{n})_n$ a $(PS)_c$ sequence,  we have, as $n\to\infty$,
$$\begin{aligned}
\int_{\mathbb R^3} &|\nabla u_n|^{p-2}\nabla u_n \nabla \psi {\rm d} x+ \int_{\mathbb R^3} |u_n|^{p-2}u_n\psi {\rm d} x+\int_{\mathbb R^3} |\nabla u_n|^{q-2}\nabla u_n \nabla \psi {\rm d} x\\ &+ \int_{\mathbb R^3} |u_n|^{q-2}u_n\psi {\rm d} x \pm \lambda\int_{\mathbb R^3}  \phi_{u_n} |u_n|^{\vartheta -2}u_n\psi {\rm d} x \\ & - \beta\int_{\mathbb R^3} \mathcal{W}(x) |u_n|^{\tau-2}u_n \psi {\rm d} x- \int_{\mathbb R^3} \mathcal{K}(x) |u_n|^{p^*-2}u_n \psi {\rm d} x=o(1).
\end{aligned}$$
Now, choosing $\psi=\psi_\varepsilon u_n$, which is still bounded, we obtain 
\begin{equation}\label{dis1}\begin{aligned}
\int_{\mathbb R^3} &\bigl(|\nabla u_n|^p+|\nabla u_n|^q\bigr)\psi_\epsilon {\rm d} x\\&+ \int_{\mathbb R^3} u_n\bigl(|\nabla u_n|^{p-2}+
|\nabla u_n|^{q-2}\bigr)\nabla u_n \nabla \psi_\epsilon {\rm d} x\\
&+ \int_{\mathbb R^3} (|u_n|^p+|u_n|^q)\psi_\epsilon {\rm d} x\pm \lambda\int_{\mathbb R^3}  \phi_{u_n} |u_n|^\vartheta \psi_\epsilon {\rm d} x  \\ &- \beta\int_{\mathbb R^3} \mathcal{W}(x) |u_n|^\tau\psi_\epsilon {\rm d} x- \int_{\mathbb R^3} \mathcal{K}(x) |u_n|^{p^*} \psi_\epsilon {\rm d} x=o(1).
\end{aligned}\end{equation}
Applying H\"older inequality 
\begin{equation}\label{6.4}
\begin{aligned}
\biggl|\int_{\mathbb R^3}&{u_n}|\nabla u_n|^{p-2}\nabla u_n\cdot\nabla \psi_\varepsilon {\rm d} x\biggr|\\&\le \| u_n\|^{p-1} \biggl(\int_{B_{\varepsilon}(x_j)}|u_n|^p|\nabla \psi_\varepsilon|^p{\rm d} x\biggr)^{1/p}.
\end{aligned}
\end{equation}
Furthermore, 
$u_n\to u$ in
$L^{p}_\mathrm{loc}(\mathbb R^3)$,  by compactness, yielding, up to subsequences
\begin{equation}\label{aeconv}
u_n(x)\to u(x)\qquad \text{a.e. \,\, in}\quad \omega=\overline{B}_{\varepsilon}(x_{j}) 
\end{equation} 
and there exists $g\in L^p(\omega, \mathbb R)$
such that $|u_{n}(x)|\le g(x)$ a.e.\ in $\omega$.  Thus, $|u_n(x)| |\nabla\psi_{\varepsilon}(x)|\le C g(x)$
a.e.\ in $\omega$, as well as in $\mathbb R^3$ being $\psi\in C_{c}^{\infty}(\mathbb R^3)$, and in turn, Lebesgue Dominated Convergence Theorem gives
\begin{equation}\label{conv_grad}
|u_n\nabla\psi_{\varepsilon}|\to|u\nabla\psi_{\varepsilon}| \ \textnormal{in} \ L^{p}(\mathbb R^3).
\end{equation}
Consequently, by \eqref{conv_grad} and
 H\"older inequality with exponents
$3/(3-p)$ and $3/p$, we obtain
$$\begin{aligned}
\lim_{n\to\infty}&\int_{B_{\varepsilon}(x_{j})}|u_n|^{p}|\nabla \psi_\varepsilon|^p {\rm d} x\\
&\le  \left(\int_{B_{\varepsilon}(x_{j})}|\nabla\psi_{\varepsilon}|^{3}{\rm d} x\right)^{p/3}
\left(\int_{B_{\varepsilon}(x_{j})}|u|^{p^*}{\rm d} x\right)^{p/p^*}\\&\le C\left(\int_{B_{\varepsilon}(x_{j})}|u|^{p^{*}}{\rm d} x
\right)^{p/p^{*}},
\end{aligned}$$
where we used that $|\nabla \psi_\varepsilon|\le C \varepsilon^{-1}$ and $|B(x_i,\varepsilon)| \le C'\varepsilon^3$.
In turn, using the boundedness of $(u_n)_n\in X$ by Lemma \ref{lembound}, then \eqref{6.4} gives
$$\limsup_{n\to\infty}\biggl|\int_{\mathbb R^3}{u_n}|\nabla u_n|^{p-2}\nabla u_n\cdot\nabla \psi_\varepsilon {\rm d} x\biggr|\le C \left(\int_{B_{\varepsilon}(x_{j})}|u|^{p^{*}}{\rm d} x
\right)^{1/p^{*}}.$$
Arguing in the same way when $p=q$ and letting $\varepsilon\to 0$ from $u\in L^{p^*}(\mathbb R^3)\cap L^{q^*}(\mathbb R^3)$, we get
\begin{equation*}
\lim_{\varepsilon\to 0}\limsup_{n\to\infty}\biggl|\int_{\mathbb R^3}u_n(|\nabla u_n|^{p-2}+|\nabla u_n|^{q-2})\nabla u_n\nabla \psi_\varepsilon {\rm d} x\biggr|=0.
\end{equation*}
Moreover, by \eqref{aeconv}
$$\biggl|\int_{\mathbb R^3} (|u_n|^p+|u_n|^q)\psi_\epsilon {\rm d} x\biggr|\le \int_{B_\varepsilon(x_j)}(|u_n|^p+|u_n|^q) {\rm d} x=o(1),$$
$$\biggl|\int_{\mathbb R^3} \mathcal{W}(x) |u_n|^\tau\psi_\epsilon {\rm d} x\biggr|=o(1),$$
as $n\to\infty$ and $\varepsilon\to0$ since $u_n\to u$ in $L^{s}_\mathrm{loc}(\mathbb R^3)$ for all $1<s<p^*$, $u\in X$ and by using the Lebesgue Dominated Convergence Theorem.

Furthermore, by \eqref{disphi}
\begin{equation}\label{stimaphiu}
\biggl|\int_{\mathbb R^3}  \phi_u |u_n|^\vartheta \psi_\epsilon {\rm d} x\biggr|\le C\|u_n\|_{W^{1,p}(B_\varepsilon(x_j))}^{\vartheta m'}=o(1)
\end{equation}
as $n\to\infty$ and $\varepsilon\to0$ by using the Lebesgue Dominated Convergence Theorem.

Then, from \eqref{dis1} and the inequalities above, we can conclude for $n$ large
\begin{equation}\label{nabla_le_K_ustar}\int_{\mathbb R^3} (|\nabla u_n|^p +|\nabla u_n|^q)\psi_\varepsilon {\rm d} x \le\int_{\mathbb R^3}\mathcal{K}(x)|u_n|^{p^*}\psi_\varepsilon {\rm d} x+o(1),\end{equation}
yelding for $\varepsilon\to0$
\begin{equation}\label{6.23}
\mu_{j}\le \mathcal{K}(x_{j})\nu_{j},
\end{equation}
since $\mathcal{K}\in C(\mathbb R^3)$ and by \eqref{ineqmeasures}. Consequently, either  $\nu_j=0$ and then also $\mu_j=0$,  or $\nu_j>0$. We claim that  
the latter case cannot occur for each $j\in J$, with $J$ given in Lemma \ref{lions}. First observe that  combining \eqref{6.23} and \eqref{6.22}, we have
\begin{equation}\label{sxj}
S\leq \mathcal{K}(x_j)\nu_j^{p/3},
\end{equation}
that is concentration of the measure $\nu$ can occur only at those points $x_j$ where $\mathcal{K}(x_j)>0$.
Consequently, from \eqref{6.22} and \eqref{6.23} the measure $\mu$ can concentrate at points
in which the measure $\nu$ can.
In turn, zeros for $\mathcal{K}$ cannot belong to $X_J=\{x_j: j\in J\}$.

Let $J_1=\{j\in J: \mathcal{K}(x_j)>0\}$, we claim that 
\begin{equation}\label{J1empty}J_1=\emptyset.\end{equation} 
If not, for any $j\in J_1$, then \eqref{sxj} implies
\begin{equation}\label{J2}
 \nu_{j}\ge \left( \frac{S}{\mathcal{K}(x_{j})}\right)^{3/p}\ge \left( \frac{S}{\|\mathcal{K}\|_\infty}\right)^{3/p},
\end{equation}
 which in particular gives that
 $|J_1|<\infty$ being $\nu$ a bounded measure; indeed, from \eqref{ineqmeasures} and \eqref{J2},
 we get 
$$\infty>\int_{\mathbb R^3}d\nu=\|u\|_{p^*}^{p^*}+\int_{\mathbb R^3}\sum_{j\in J_1}\nu_j\delta_{x_j}{\rm d} x+\nu_\infty
\ge \|u\|_{p^*}^{p^*}+\left(\frac{S}{\|K\|_\infty}\right)^{3/p}|J_1|+\nu_\infty.$$
 Now, we show that \eqref{J2} cannot occur. 
 Now we divide the proof into cases

{\it Case 1.}  If $\vartheta m'\le\tau<p^*\,$ and $\, \max\{1,q/(m^*)',p/m'\}<\vartheta<p^*/m'$, then  
\begin{equation}\label{epn2+-}\begin{aligned}
&c+o(1)=\mathcal{J}(u_n)-\frac{1}{\vartheta m'}\mathcal{J}'(u_n)u_n\\&=\left(\frac{1}{p}-\frac{1}{\vartheta m'}\right)\|u_n\|_{1,q}^p+\left(\frac{1}{q}-\frac{1}{\vartheta m'}\right)\|u_n\|_{1,q}^q\\&\,\,\,\,-\beta \left(\frac{1}{\tau}-\frac{1}{\vartheta m'}\right)\int_{\mathbb R^3} \mathcal{W}|u_n|^\tau {\rm d} x+\left(\frac{1}{\vartheta m'}-\frac{1}{p^*}\right)\int_{\mathbb R^3} \mathcal{K}(x)|u_n|^{p^*}{\rm d} x
\\&\ge \left(\frac{1}{\vartheta m'}-\frac{1}{p^*}\right)\int_{\mathbb R^3} \mathcal{K}(x)|u_n|^{p^*}{\rm d} x\\&\ge \left(\frac{1}{\vartheta m'}-\frac{1}{p^*}\right)\int_{B_{\varepsilon}(x_{j})} \mathcal{K}(x)|u_n|^{p^*}{\rm d} x,
\end{aligned}\end{equation}
for any $\varepsilon>0$. In particular, \eqref{epn2+-} holds both for $\mathcal{J}_+$ and for  $\mathcal{J}_-$.
Consequently,  inserting  \eqref{J2}  in \eqref{epn2+-} and letting $n\to\infty$ and $\varepsilon\to0$, both for $\mathcal{J}_+$ and for  $\mathcal{J}_-$, we obtain
$$c\ge \left(\frac{1}{\vartheta m'}-\frac{1}{p^*}\right) \nu_j \mathcal{K}(x_j)\ge \left(\frac{1}{\vartheta m'}-\frac{1}{p^*}\right)\frac{S^{3/p}}{\|\mathcal{K}\|_\infty^{3/p^*}}(=c_{PS}'),$$
yielding a contradiction, so that the claim \eqref{J1empty} is proved. However, this case cannot occur for $\mathcal{J}_-$ since Lemma \ref{lembound} does not hold if $\vartheta m'\le\tau<p^*$, cfr \eqref{cond-}$_1$.

{\it Case 2.} If $p\le\tau<p^*$, then 
\begin{equation*}\begin{aligned}c+o(1)&=\mathcal{J}(u_n)-\frac{1}{p}\mathcal{J}'(u_n)u_n\\&=\left(\frac{1}{q}-\frac{1}{p}\right)\|u_n\|_{1,q}^q\pm \frac{\lambda}{ \vartheta}\left(\frac{1}{m'}-\frac{\vartheta}{p}\right)\int_{\mathbb R^3} \phi_{u_n} |u_n|^{\vartheta}{\rm d} x\\&\qquad+\beta\left(\frac{1}{p}-\frac{1}{\tau}\right)\int_{\mathbb R^3} \mathcal{W}(x)|u_n|^{\tau}{\rm d} x+\frac{1}{3}\int_{\mathbb R^3} \mathcal{K}(x)|u_n|^{p^*}{\rm d} x
\\&\ge \frac 13\int_{\mathbb R^3} \mathcal{K}(x)|u_n|^{p^*}{\rm d} x\pm \frac{\lambda}{ \vartheta}\left(\frac{1}{m'}-\frac{\vartheta}{p}\right)\int_{\mathbb R^3} \phi_{u_n} |u_n|^{\vartheta}{\rm d} x,
\end{aligned}\end{equation*}
so that, if $\,\max\{1,q/(m^*)'\}<\vartheta<\min\{p^*/(m^*)',p/m'\}$, we have
$$c +o(1)=\mathcal{J}_+(u_n)-\frac{1}{p}\mathcal{J}_+'(u_n)u_n\ge \frac 13\int_{\mathbb R^3} \mathcal{K}(x)|u_n|^{p^*}{\rm d} x$$
or, if $\max\{1,q/(m^*)',p/m'\}<\vartheta<p^*/(m^*)'$, we have 
$$c +o(1)=\mathcal{J}_-(u_n)-\frac{1}{p}\mathcal{J}_-'(u_n)u_n\ge \frac 13\int_{\mathbb R^3} \mathcal{K}(x)|u_n|^{p^*}{\rm d} x.$$ 
In turn, arguing as in {\it Case} 1, both for $\mathcal{J}_+$ and for  $\mathcal{J}_-$, thanks to \eqref{J2} it holds
$$c\ge \frac 13 \nu_j \mathcal{K}(x_j)\ge \frac{S^{3/p}}{3\|\mathcal{K}\|_\infty^{3/p^*}}(=c_{PS}),$$
yielding a contradiction, so that also in {\it Case} 2 the claim \eqref{J1empty} is proved

We have so obtained that in both cases concentration cannot occur at finite points.

It remains to show that the concentration
of $\nu$ cannot occur at infinity, namely $\nu_\infty=0$. 
We use the same idea employed to prove \eqref{6.23}, but with the following cutoff function   $\psi_{R}\in C^{\infty}(\mathbb R^3)$
such that $0\le\psi_{R}\le1$ in $\mathbb R^{3}$, $\psi_R(x)=0$ for $|x|<R$ and $\psi_{R}(x)=1$ for $|x|>2R$. In this way, we obtain \eqref{nabla_le_K_ustar} with $\psi_\varepsilon$ replaced by $\psi_R$, that is 
$$\int_{B_R^c} (|\nabla u_n|^p +|\nabla u_n|^q)\psi_R {\rm d} x \le\int_{B_R^c}\mathcal{K}(x)|u_n|^{p^*}\psi_R {\rm d} x+o(1)$$
yielding $\mu_\infty\le \|\mathcal{K}\|_\infty\nu_\infty$ since
$$\lim_{R\to\infty}\limsup_{n\to\infty}\left\{\int_{B_R^c}\mathcal{K}|u_{n}|^{p^{*}}\psi_{R}{\rm d} x\right\}
\le\|\mathcal{K}\|_{\infty}\nu_{\infty},$$
by \eqref{munuinf}, which gives $\mu_\infty\le \|\mathcal{K}\|_\infty\nu_\infty$. Moreover, recalling \eqref{6.22} we arrive at $S\le \nu_{\infty}^{p/3}\|\mathcal{K}\|_\infty$.
Arguing as in {\it Case 1} and {\it Case 2} above with $B_\varepsilon(x_j)$ replaced by $B_R^c$, we get again a contradiction. 

In turn, $\nu_\infty=\mu_\infty=0$ is in force. 
Consequently, recalling also that $\nu_i=\mu_i=0$, by \eqref{nunuinfty} we end up with
$$\lim_{n\to\infty}\int_{\mathbb R^3}|u_{n}|^{p^{*}}{\rm d} x=\int_{\mathbb R^3}|u|^{p^{*}}{\rm d} x,$$
that is $\|u_n\|_{p^*}\to\|u\|_{p^*}$ as $n\to\infty$, which combined with $u_n(x)\to u(x)$ a.e. in $\mathbb R^3$, the latter obtained by \eqref{aeconv} and an exhaustion process applied to a.e. convergence on compact sets in $\mathbb R^3$,  implies $\|u_n-u\|_{p^*}\to 0$ by Brezis Lieb Lemma in \cite{BL1983}. It remains to prove 
\begin{equation}\label{claimfin}
\|u_n-u\|\to0, \quad \text{as}\,\,\,n\to\infty,
\end{equation}
 which is equivalent to prove
\begin{equation*}
\int_{\mathbb R^3}|\nabla (u_n-u)|^p+|u_n-u|^p{\rm d} x, \,\, \int_{\mathbb R^3}|\nabla (u_n-u)|^q + |u_n-u|^q {\rm d} x\to 0 
\quad\mbox{as}\quad n\to\infty.\end{equation*}
To this aim, since $(u_n)_n$ is a $(PS)_c$ sequence, we have
\begin{equation}\label{g-11}\begin{aligned}
o(1)&=\langle \mathcal{J}'(u_n)-\mathcal{J}'(u),u_n-u\rangle \\&=\int_{\mathbb R^3} \bigl(|\nabla u_n|^{p-2}\nabla u_n-|\nabla u|^{p-2}\nabla u)(\nabla (u_n-u)\bigr){\rm d} x\\
&\,\,\,\,\,+\int_{\mathbb R^3} \bigl(|\nabla u_n|^{q-2}\nabla u_n-|\nabla u|^{q-2}\nabla u)(\nabla (u_n-u)\bigr){\rm d} x\\
&\,\,\,\,\,+\int_{\mathbb R^3}(|u_n|^{p-2}u_n-|u|^{p-2}u)(u_n-u){\rm d} x\\&\,\,\,\,\,+\int_{\mathbb R^3} (|u_n|^{q-2}u_n-|u|^{q-2}u)(u_n-u){\rm d} x
\\ &\,\,\,\,\,\pm \lambda\int_{\mathbb R^3} \phi_u(|u_n|^{\vartheta-2}u_n-|u|^{\vartheta-2}u)(u_n-u){\rm d} x
\\&\,\,\,\,\,-\beta\int_{\mathbb R^3}\mathcal{W}(x)(|u_n|^{\tau-2}u_n-|u|^{\tau-2}u)(u_n-u){\rm d} x
\\&\,\,\,\,\,-\int_{\mathbb R^3}\mathcal{K}(x)(|u_n|^{p^*-2}u_n-|u|^{p^*-2}u)(u_n-u){\rm d} x.
\end{aligned}\end{equation}
By Lemma \ref{lembound}, using H\"older and Schwarz inequality, the $L^{p^*}$ function estimates of $(u_n)_n$ and the convergence of $(u_n)_n$ in $L^{p^*}(\mathbb R^3)$ we get
 
$$\begin{aligned}
\lefteqn{\left|\int_{\mathbb R^3}\mathcal{K}(x)(|u_n|^{p^*-2}u_n-|u|^{p^*-2}u)(u_n-u) {\rm d} x\right|}
\\
&\le \|\mathcal{K}\|_\infty \int_{\mathbb R^3} \left(|u_n|^{p^*-1}+|u|^{p^*-1}\right)|u_n-u|{\rm d} x\\
& \le \|\mathcal{K}\|_\infty \left(\|u_n\|_{p^*}^{p^*-1}+\|u\|_{p^*}^{p^*-1}\right)\|u_n-u\|_{p^*}=o(1).
\end{aligned}$$
Similarly, we have
$$\begin{aligned}
&\left|\int_{\mathbb R^3}\mathcal{W}(x)(|u_n|^{\tau-2}u_n-|u|^{\tau-2}u)(u_n-u){\rm d} x\right|=o(1),
\end{aligned}$$
and arguing as in \eqref{stimaphiu}
$$\left|\int_{\mathbb R^3} \phi_u(|u_n|^{\vartheta-2}u_n-|u|^{\vartheta-2}u)(u_n-u){\rm d} x\right|=o(1).$$
Thus, \eqref{g-11} reduces to 
\begin{equation*}\begin{aligned}
o(1)&=\langle \mathcal{J}'(u_n)-\mathcal{J}'(u),u_n-u\rangle\\
&=\int_{\mathbb R^3} \bigl(|\nabla u_n|^{p-2}\nabla u_n-|\nabla u|^{p-2}\nabla u)(\nabla (u_n-u)\bigr){\rm d} x\\
&+\int_{\mathbb R^3} \bigl(|\nabla u_n|^{q-2}\nabla u_n-|\nabla u|^{q-2}\nabla u)(\nabla (u_n-u)\bigr){\rm d} x\\
&+\int_{\mathbb R^3}(|u_n|^{p-2}u_n-|u|^{p-2}u)(u_n-u){\rm d} x\\&+\int_{\mathbb R^3} (|u_n|^{q-2}u_n-|u|^{q-2}u)(u_n-u){\rm d} x,
\end{aligned}\end{equation*}
so that by Simon inequality, see \cite{simon}, valid for all $s>1$ and
$a,b\in\mathbb R^3$
\begin{equation*}
|a\!-\!b|^{s}\!\lesssim\!\!\begin{cases}(|a|^{s-2}a- |b|^{s-2}b)(a-b)&\phantom{1<\,}s\!\geq \!2;\\
 \!\left(|a|^{s-2}a- |b|^{s-2}b)(a-b\!\right)\!^{\frac{s}{2}}\!\left(|a|^s+|b|^s\right)^{\frac{2-s}{2}} 
&1\!<\!s\!<\!2,\end{cases}\end{equation*}
as in pag 713 in \cite{FPR}, then \eqref{claimfin} is in force concluding the proof of the lemma.
\end{proof}

The final step consists in showing that the mountain pass levels $c_u^{\pm}$, defined in \eqref{cu+-}, remain below the thresholds $c_{PS}$ and $c_{PS}'$ from \eqref{Csegnato} and \eqref{Csegnato'}. Below these levels, the functional $\mathcal{J}$ satisfies the Palais-Smale condition, as established in the previous lemma, thereby restoring compactness.

To this aim, from now on we denote, for each $\lambda, \beta>0$,
\begin{equation}\label{clambda+-}\begin{aligned}
&\hat c_+ := \inf_{u\in X\setminus\{0\}} \max_{t\ge 0} \mathcal{J}_+(tu),\qquad \hat c_- := \inf_{u\in X\setminus\{0\}} \max_{t\ge 0} \mathcal{J}_-(tu).
\end{aligned}\end{equation}
Note that $\hat c_{\pm}\geq c_u^{\pm}$ since $\mathcal{J}_\pm(tu)<0$ for $u\in X\setminus\{0\}$ and $t$ large by the structure of $\mathcal{J}_\pm$, as observed in Theorem 4.2 in \cite{Willemminimax}.

\begin{lemma}\label{c<csegnato}
Assume \eqref{K}, \eqref{W}, \eqref{defthetamain} and $p\le  \tau <p^*$.
Let $c_{PS}$ and $\hat c_{\pm}$ defined in \eqref{Csegnato} and \eqref{clambda+-}, respectively.
 Then, there exists $\lambda^*, \beta^*>0$ such that
 \begin{itemize}
     \item $0< \hat c_+<c_{PS}$ for all $\beta>\beta^*$ and $\lambda>0$ provided that 
     $$\max\left\{1,\frac{q}{(m^*)'}\right\}<\vartheta<\frac{p^*}{m'}.$$
     \item $0< \hat c_-<c_{PS}$
     \begin{itemize}
         \item[$(a)$] for all $\beta>\beta^*$ and $\lambda>0$ provided that $$p \le\tau< \min\{\vartheta m', p^*\}\quad\text{and}\quad \max\left\{1,\dfrac{q}{(m^*)'},\dfrac{p}{m'}\right\}<\vartheta<\dfrac{p^*}{(m^*)'}$$
         \item[$(b)$] for all $\lambda>\lambda^*$ and $\beta>0$ provided that 
         $$\max\{\vartheta m',p\}\le \tau<p^*\quad\text{and}\quad \max\left\{1,\dfrac{q}{(m^*)'}\right\}<\vartheta<\dfrac{p^*}{m'}.$$
     \end{itemize}
 \end{itemize}
\end{lemma}

\begin{remark} Before proceeding with the proof, we present the following observations:
\begin{itemize}
    \item We point out that the case $\tau=\vartheta m'$ is handled both in $(a)$ and $(b)$, yielding the validity of $0< \hat c_-<c_{PS}$ either for $\beta$ large or for $\lambda$ large.
\item As it will be clear from the proof below, since the upper bound for $\hat c_+$ is obtained by a limit procedure, then Lemma \ref{c<csegnato} continues to be valid with $c_{PS}$ replaced by $c_{PS}'$.
\item In view of \eqref{cond-}, case (b) is excluded due to the incompatibility of the ranges for $\tau$. Consequently, existence in Theorem \ref{th1} is established only for sufficiently large values of $\beta$.
\end{itemize}
\end{remark}

\begin{proof}[Proof of Lemma \ref{c<csegnato}]
Take the open set $\Omega_\mathcal{W}$ where $\mathcal{W}$ is positive by \eqref{W}. Let $u_0\in X\setminus \{0\}$ with 
$|\mbox{supp}(u_0)\cap\Omega_{\mathcal{W}}|>0$ such that $u_0\ge0$ and $\|\nabla u_0\|_q>0$. Take any $t\ge 0$
$$\begin{aligned}
\mathcal{J}(tu_0)=\frac{t^p}{p}\|u_0\|_{1,p}^p&+\frac{t^q}{q}\|u_0\|_{1,q}^q\pm \frac{\lambda}{ \vartheta m'}t^{\vartheta m'}\int_{\mathbb R^3} \phi_{u_0} |u_0|^{\vartheta}{\rm d} x\\&-\beta\frac{t^\tau}{\tau}\int_{\mathbb R^3} \mathcal{W}(x)|u_0|^{\tau}{\rm d} x-\frac{t^{p^*}}{p^*}\int_{\mathbb R^3} \mathcal{K}(x)|u_0|^{p^*}{\rm d} x.
\end{aligned}$$
Now, from $q<p\le \tau<p^*$ and either by \eqref{defthetamain} in case $\mathcal{J}_-$ or by $\max\{1,q/(m^*)'\}<\vartheta<p^*/m'$ in case $\mathcal{J}_+$, 
it follows that
$\mathcal{J}(tu_0)\to 0^+$ as $t\to0^+$ being positive the coefficient of the lower order term in $t$, while $\mathcal{J}(tu_0)\to-\infty$ as $t\to\infty$, being negative the coefficient of the higher order term.
Thus, there exists $t_{\lambda,\beta}>0$ 
such that
$$\max_{t\ge 0}\mathcal{J}(tu_0)=\mathcal{J}(t_{\lambda,\beta} u_0).$$
In particular, being $\mathcal{J}\in C^1$, we get
\begin{equation}\label{dtjt0}\begin{aligned}0&=\frac{d}{dt}\Bigl[
\mathcal{J}(tu_0)\Bigr]_{t=t_{\lambda,\beta}}\\&=t_{\lambda,\beta}^{p-1}\|\nabla u_0\|_p^p+t_{\lambda,\beta}^{q-1}\|\nabla u_0\|_q^q
\pm \lambda t^{\vartheta m'-1}\int_{\mathbb R^3} \phi_{u_0} u_0^{\vartheta}{\rm d} x\\&\quad -\beta t_{\lambda,\beta}^{\tau-1}\int_{\mathbb R^3}\mathcal{W}(x)u_0^{\tau}\, {\rm d} x -t_{\lambda,\beta}^{p^*-1}\int_{\mathbb R^3}\mathcal{K}(x)  u_0^{p^*}
\,{\rm d} x.\end{aligned}\end{equation}
Now we divide the proof into cases.

{\it Case $\mathcal{J}_+$:} In this case \eqref{dtjt0} is equivalent to
\begin{equation}\label{primozero}\begin{aligned}
\frac{\|\nabla u_0\|_p^p}{t_{\lambda,\beta}^{\tau -p}}+\frac{\|\nabla u_0\|_q^q}{t_{\lambda,\beta}^{\tau-q}}
 &-t_{\lambda,\beta}^{p^*-\tau}\int_{\mathbb R^3}\mathcal{K}(x)  u_0^{p^*} \,{\rm d} x
+ \frac{\lambda}{t_{\lambda,\beta}^{\tau-\vartheta m'}}\int_{\mathbb R^3} \phi_{u_0} u_0^{\vartheta}{\rm d} x\\&=\beta \int_{\mathbb R^3}\mathcal{W}(x)u_0^{\tau}\, {\rm d} x\end{aligned}\end{equation}
 for every $\lambda, \beta>0$. Since the support of $u_0$ is contained in $\Omega_\mathcal{W}$, the right hand side of
\eqref{primozero} is positive and it goes to $\infty$ if $\beta\to\infty$. Thus, also the left hand side of
\eqref{primozero} must go to $\infty$ if $\beta\to\infty$ and for all $\lambda$. Being $q<p\le \tau<p^*$, necessarily $t_{\lambda,\beta}\to 0^+$ as $\beta\to\infty$
and for all $\lambda>0$ since 
$$\begin{aligned}
\frac{\|\nabla u_0\|_p^p}{t_{\lambda,\beta}^{\tau -p}}&+\frac{\|\nabla u_0\|_q^q}{t_{\lambda,\beta}^{\tau-q}}
 -t_{\lambda,\beta}^{p^*-\tau}\int_{\mathbb R^3}\mathcal{K}(x)  u_0^{p^*} \,{\rm d} x
+ \frac{\lambda}{t_{\lambda,\beta}^{\tau-\vartheta m'}}\int_{\mathbb R^3} \phi_{u_0} u_0^{\vartheta}{\rm d} x\\&
\ge \frac{\|\nabla u_0\|_q^q}{t_{\lambda,\beta}^{\tau-q}}
 -t_{\lambda,\beta}^{p^*-\tau}\int_{\mathbb R^3}\mathcal{K}(x)  u_0^{p^*} \,{\rm d} x \sim \frac{\|\nabla u_0\|_q^q}{t_{\lambda,\beta}^{\tau-q}} \to\infty,
\end{aligned}$$
as $\beta\to\infty$.
From $\mathcal{J}(t_{\lambda,\beta} u_0)\to 0^+$ as $t_{\lambda,\beta}\to0^+$ or equivalently when $\beta\to\infty$,
 we can conclude that there exists $\beta^*>0$ such that for all $\beta>\beta^*$ and $\lambda>0$
\begin{equation}\label{maxj+}
\max_{t\ge0}\mathcal{J}_+(tu_0)=\mathcal{J}_+(t_{\lambda,\beta} u_0)<c_{PS}.
\end{equation}
By the definition of $\hat c_+$, we get $\hat c_+<c_{PS}$
for all $\beta>\beta^*$ and $\lambda>0$.

{\it Case $\mathcal{J}_-$:} We split the the interval for $\tau$ in two disjoint intervals. 

If $p \le\tau< \vartheta m'$ and $\max\{1,q/(m^*)',p/m'\}<\vartheta<p^*/(m^*)'$, we use \eqref{primozero} with $\lambda$ replaced with $-\lambda$, to deduce that $t_{\lambda,\beta}\to 0^+$ as $\beta\to\infty$
and for all $\lambda>0$ since
$$\begin{aligned}
\frac{\|\nabla u_0\|_p^p}{t_{\lambda,\beta}^{\tau -p}}&+\frac{\|\nabla u_0\|_q^q}{t_{\lambda,\beta}^{\tau-q}}
 -t_{\lambda,\beta}^{p^*-\tau}\int_{\mathbb R^3}\mathcal{K}(x)  u_0^{p^*} \,{\rm d} x
- t_{\lambda,\beta}^{\vartheta m'-\tau}\lambda\int_{\mathbb R^3} \phi_{u_0} u_0^{\vartheta}{\rm d} x\\&
\ge \frac{\|\nabla u_0\|_q^q}{t_{\lambda,\beta}^{\tau-q}}
 -t_{\lambda,\beta}^{p^*-\tau}\int_{\mathbb R^3}\mathcal{K}(x)  u_0^{p^*} \,{\rm d} x- t_{\lambda,\beta}^{\vartheta m'-\tau}\lambda\int_{\mathbb R^3} \phi_{u_0} u_0^{\vartheta}{\rm d} x\\& \sim \frac{\|\nabla u_0\|_q^q}{t_{\lambda,\beta}^{\tau-q}} \to\infty,
\end{aligned}$$
as $\beta\to\infty$. Thus, arguing as in \eqref{maxj+} replacing $\mathcal{J}_+$ with respectively $\mathcal{J}_-$, we get $\hat c_-<c_{PS}$
for all $\beta>\beta^*$ and $\lambda>0$.

While, if $\vartheta m'\le \tau<p^*$ and $\max\{1,q/(m^*)'\}<\vartheta<p^*/m'$, we write \eqref{dtjt0} as follows
\begin{equation}\label{primozero'}\begin{aligned}
\frac{\|\nabla u_0\|_p^p}{t_{\lambda,\beta}^{\vartheta m' -p}}&+\frac{\|\nabla u_0\|_q^q}{t_{\lambda,\beta}^{\vartheta m'-q}}
 -t_{\lambda,\beta}^{p^*-\vartheta m'}\int_{\mathbb R^3}\mathcal{K}(x)  u_0^{p^*} \,{\rm d} x
 \\&-\beta t_{\lambda,\beta}^{\tau-\vartheta m'}\int_{\mathbb R^3}\mathcal{W}(x)u_0^{\tau}{\rm d} x= \lambda \int_{\mathbb R^3} \phi_{u_0} |u_0|^{\vartheta}{\rm d} x
 \end{aligned}\end{equation}
for every $\beta,\lambda>0$. Being $u_0$ is nontrivial, then the right hand side of
\eqref{primozero'} is  positive and goes to $\infty$ if $\lambda\to\infty$. Thus, also the left hand side of
\eqref{primozero'} must go to $\infty$ if $\lambda\to\infty$, this occurs if $t_{\lambda,\beta}\to 0^+$ as $\lambda\to\infty$ being $\vartheta>q/(m^*)'>q/m'$ and $\|\nabla u_0\|_q>0$ since
$$\begin{aligned}
\frac{\|\nabla u_0\|_p^p}{t_{\lambda,\beta}^{\vartheta m' -p}}&+\frac{\|\nabla u_0\|_q^q}{t_{\lambda,\beta}^{\vartheta m'-q}}
 -t_{\lambda,\beta}^{p^*-\vartheta m'}\int_{\mathbb R^3}\mathcal{K}(x)  u_0^{p^*} \,{\rm d} x
 -\beta t_{\lambda,\beta}^{\tau-\vartheta m'}\int_{\mathbb R^3}\mathcal{W}(x)u_0^{\tau}{\rm d} x\\&
 \ge \frac{\|\nabla u_0\|_q^q}{t_{\lambda,\beta}^{\vartheta m'-q}}
 -t_{\lambda,\beta}^{p^*-\vartheta m'}\int_{\mathbb R^3}\mathcal{K}(x)  u_0^{p^*} \,{\rm d} x
 -\beta t_{\lambda,\beta}^{\tau-\vartheta m'}\int_{\mathbb R^3}\mathcal{W}(x)u_0^{\tau}{\rm d} x \\&\sim \frac{\|\nabla u_0\|_q^q}{t_{\lambda,\beta}^{\vartheta m'-q}} \to\infty
\end{aligned}$$
as $\lambda\to\infty$. From $\mathcal{J}(t_{\lambda,\beta} u_0)\to 0^+$ as $t_{\lambda,\beta}\to0^+$ or equivalently when $\lambda\to\infty$,
 we can conclude that there exists $\lambda^*>0$ such that for all $\lambda>\lambda^*$ and $\beta>0$
\begin{equation*}
\max_{t\ge0}\mathcal{J}_-(tu_0)=\mathcal{J}_-(t_{\lambda,\beta} u_0)<c_{PS}.
\end{equation*}
By the definition of $\hat c_-$, we get $\hat c_-<c_{PS}$
for all $\lambda>\lambda^*$ and $\beta>0$.
\end{proof}

Finally, we have all the ingredients to conclude the proof of the main theorem of the paper.

{\it Proof of Theorem \ref{th1}.} Under the different assumptions on $\tau,\vartheta$ in the statement of Theorem \ref{th1}, Lemma \ref{pl106} implies that the energy functional $\mathcal{J}$ has the mountain pass geometry. Thus, Theorem \ref{mptheorem} can be applied giving the existence of a Palais-Smale sequence $(u_n)_n\subset X$ at level $c_u$. By employing the boundedness of such sequence given by Lemma \ref{lembound}, we proved a compactness assumption in Lemma \ref{pscomp} in terms of the validity of the Palais-Smale condition for suitable levels, and finally Lemma \ref{c<csegnato} confirms that $(u_n)_n\subset X$ converges, up to subsequences, to a nontrivial function $u\in X$ having positive energy and for which it holds $\mathcal{J}'(u_n)\varphi=0$ for all $\varphi\in X$, i.e. $u$ is a weak solution to \eqref{maineq}.

\section*{Acknowledgments}

The authors are members of the {\em Gruppo Nazionale per l'Analisi Ma\-te\-ma\-ti\-ca, la Probabilit\`a e le loro Applicazioni} (GNAMPA) of the {\em Istituto Nazionale di Alta Matematica} (INdAM). L.B is partially supported by INdAM-GNAMPA Project 2026 titled \textit{Structural degeneracy and criticality in (sub)elliptic PDEs} (E53C25002010001) and by the Deutsche Forschungsgemeinschaft (DFG, German Research Foundation) - Project-ID 258734477 - SFB 1173.


\begin{thebibliography}{99}

\bibitem{AmbrosettiMilan} A. Ambrosetti, {\it On Schr\"odinger-Poisson systems},  Milan J. Math  {\bf 76} (2008), 257--274.

\bibitem{AR} A. Ambrosetti and P. H. Rabinowitz, {\it Dual variational methods in critical point theory and applications},  J. Functional
Analysis  {\bf 14} (1973), 349--381.

\bibitem{amru} A. Ambrosetti and D. Ruiz, {\it Multiple bound states for the Schr\"odinger-Poisson problem},  Commun. Contemp. Math.  {\bf 10} (2008), 391--404.

\bibitem{AAP} A. Azzollini, P. d’Avenia, and A. Pomponio, {\it On the Schr\"odinger-Maxwell equations under the effect of a general
nonlinear term} Ann. Inst. H. Poincar\'e C Anal. Non Lin\'eaire {\bf 27}(2010) 779--791.

\bibitem{AP} A. Azzollini and A. Pomponio, {\it Ground state solutions for the nonlinear Schr\"odinger-Maxwell equations} J. Math.
Anal. Appl. {\bf 345} (2008) 90--108.

\bibitem{analisimoderna} M. Badiale and E. Serra, {\it Semilinear elliptic equations for beginners}, Universitext Springer, London, 2011.

\bibitem{BBF} L. Baldelli, Y. Brizi, and R. Filippucci, {\it Multiplicity results for $(p, q)$-Laplacian equations with critical exponent in $\mathbb R^N$ and negative energy} Calc. Var. Partial Differential Equations {\bf 60} (2021) p. 30.

\bibitem{BFccm} L. Baldelli and R. Filippucci, {\it Existence of solutions for critical $(p, q)$-Laplacian equations in $\mathbb R^N$}
 Commun. Contemp. Math. {\bf 25} (2023) p. 26.

\bibitem{BNTW} A. K. Ben-Naoum, C. Troestler, and M. Willem, {\it Extrema problems with critical Sobolev exponents on unbounded
domains} Nonlinear Anal. {\bf 26} (1996) 823--833.

\bibitem{BDAFP00} V. Benci, P. D’Avenia, D. Fortunato, and L. Pisani, {\it Solitons in several space dimensions: Derrick’s problem and
infinitely many solutions} Arch. Ration. Mech. Anal. {\bf 154} (2000) 297--324.

\bibitem{articolo_6_1} V. Benci and D. Fortunato, {\it An eigenvalue problem for the Schr\"odinger-Maxwell equations} Topol. Methods Nonlinear Anal. {\bf 11} (1998) 283--293.

\bibitem{articolo_7_1} V. Benci and D. Fortunato, {\it Solitary waves of the nonlinear Klein-Gordon equation coupled with the Maxwell equations} Rev. Math. Phys. {\bf 14} (2002) 409--420.

\bibitem{BInat} M. Born and L. Infeld, {\it Foundations of the new field theory} Nature {\bf 132} (1933).

\bibitem{BI} M. Born and L. Infeld, {\it Foundations of the new field theory} Proc. Roy. Soc. London Ser. A {\bf 144} (1934) 425--451.

\bibitem{Brezisaf} H. Brezis, {\it Functional analysis, Sobolev spaces and partial differential equations} Universitext Springer, New York, 2011.

\bibitem{BL1983} H. Brezis and E. Lieb, {\it A relation between pointwise convergence of functions and convergence of functionals} Proc. Amer. Math. Soc. {\bf 88} (1983) 486--490.

\bibitem{BN} H. Brezis and L. Nirenberg, {\it Positive solutions of nonlinear elliptic equations involving critical Sobolev exponents}
Comm. Pure Appl. Math. {\bf 36} (1983) 437--477.

\bibitem{CDM} G. Caristi, L. D’Ambrosio, and E. Mitidieri, {\it Representation formulae for solutions to some classes of higher order
systems and related Liouville theorems} Milan J. Math. {\bf 76} (2008) 27--67.

\bibitem{CLR} D. Cassani, Z. Liu, and G. Romani, {\it Nonlocal planar Schr\"odinger-Poisson systems in the fractional Sobolev limiting
case} J. Differential Equations {\bf 383} (2024) 214--269.

\bibitem{GC} D. Cassani, Z. Liu, and G. Romani, {\it Nonlocal Schr\"odinger-Poisson systems in $\mathbb R^N$: the fractional Sobolev limiting
case} Rend. Istit. Mat. Univ. Trieste {\bf 57} (2025) p. 33.

\bibitem{CW} S. Cingolani and T. Weth, {\it On the planar Schr\"odinger-Poisson system} Ann. Inst. H. Poincar\'e C Anal. Non Lin\'eaire {\bf 33} (2016) 169--197.

\bibitem{coc} G. M. Coclite, {\it A multiplicity result for the nonlinear Schr\"odinger-Maxwell equations} Commun. Appl. Anal. {\bf 7} (2003) 417--423.

\bibitem{DMP1} L. D’Ambrosio, E. Mitidieri, and S. I. Pohozaev, {\it Representation formulae and inequalities for solutions of a class
of second order partial differential equations} Trans. Amer. Math. Soc. {\bf 358} (2006) 893--910.

\bibitem{aprile} T. D’Aprile, {\it Semiclassical states for the nonlinear Schr\"odinger equation with the electromagnetic field} NoDEA Nonlinear Differential Equations Appl, {\bf 13} (2007) 655--681.

\bibitem{aprileW} T. D’Aprile and J. Wei, {\it On bound states concentrating on spheres for the Maxwell-Schr\"odinger equation} SIAM J. Math. Anal. {\bf 37} (2005) 321--342.

\bibitem{avenia} P. d’Avenia, {\it Non-radially symmetric solutions of nonlinear Schr\"odinger equation coupled with Maxwell equations}
Adv. Nonlinear Stud. {\bf 2} (2002) 177--192.

\bibitem{gDeK} G. H. Derrick, {\it Comments on nonlinear wave equations as models for elementary particles} J. Mathematical Phys. {\bf 5} (1964) 1252--1254.

\bibitem{DH} P. Dr\'abek and Y. X. Huang, {\it Multiplicity of positive solutions for some quasilinear elliptic equation in $\mathbb R^N$ with critical Sobolev exponent} J. Differential Equations {\bf 140} (1997) 106--132.

\bibitem{DSW2022} Y. Du, J. Su, and C. Wang, {\it On the critical Schr\"odinger-Poisson system with $p$-Laplacian} Commun. Pure Appl. Anal. {\bf 21} (2022) 1329--1342.

\bibitem{DSW23} Y. Du, J. Su, and C. Wang, {\it The quasilinear Schr\"odinger-Poisson system} J. Math. Phys. {\bf 64} (2023) p.20.

\bibitem{FPR} R. Filippucci, P. Pucci, and V. R\v adulescu, {\it Existence and non-existence results for quasilinear elliptic exterior problems with nonlinear boundary conditions} Comm. Partial Differential Equations {\bf 33} (2008) 706--717.

\bibitem{GV88T} M. Guedda and L. V\'eron, {\it Bifurcation phenomena associated to the $p$-Laplace operator} Trans. Amer. Math. Soc.
{\bf 310} (1988) 419--431.

\bibitem{GV88} M. Guedda and L. V\'eron, {\it Local and global properties of solutions of quasilinear elliptic equations} J. Differential Equations, {\bf 76} (1988) 159--189.

\bibitem{GV89} M. Guedda and L. V\'eron, {\it Quasilinear elliptic equations involving critical Sobolev exponents} Nonlinear Anal. {\bf 13} (1989) 879--902.

\bibitem{HIT} J. Hirata, N. Ikoma, and K. Tanaka, {\it Nonlinear scalar field equations in $\mathbb R^N$: mountain pass and symmetric mountain pass approaches} Topol. Methods Nonlinear Anal. {\bf 35} (2010) 253--276.

\bibitem{HS24} L. Huang and J. Su, {\it Multiple positive solutions of the quasilinear Schr\"odinger-Poisson system with critical exponent in $D^{1,p}(\mathbb R^3)$} J. Math. Phys. {\bf 65} (2024) p. 16.

\bibitem{J} L. Jeanjean, {\it Existence of solutions with prescribed norm for semilinear elliptic equations} Nonlinear Anal.
{\bf 28} (1997) 1633--1659.

\bibitem{JLC} L. Jeanjean and S. Le Coz, {\it An existence and stability result for standing waves of nonlinear Schr\"odinger equations} Adv. Differential Equations {\bf 11} (2006) 813--840.

\bibitem{LLS} Y. Li, F. Li, and J. Shi, {\it Existence of a positive solution to Kirchhoff type problems without compactness conditions}
J. Differential Equations {\bf 253} (2012) 2285--2294.

\bibitem{LPNX} S. Liang, P. Pucci, T. V. Nguyen, and D. Xiao, {\it Multiplicity and concentration of normalized solutions for double
critical Schr\"odinger-Poisson systems involving the fractional $p$-Laplacian in $\mathbb R^3$} Nonlinear Anal. {\bf 269} (2026) p. 30.

\bibitem{L84} P.-L. Lions, {\it The concentration-compactness principle in the calculus of variations. The locally compact case. II} Ann. Inst. H. Poincar\'e Anal. Non Lin\'eaire {\bf 1} (1984) 223--283.

\bibitem{L3} P.-L. Lions, {\it The concentration-compactness principle in the calculus of variations. The limit case. I} Rev. Mat. Iberoamericana, {\bf 1} (1985) 145--201.

\bibitem{MP1} E. Mitidieri and S. I. Pokhozhaev, {\it The positivity property of solutions of some nonlinear elliptic inequalities in $\mathbb R^n$} Dokl. Akad. Nauk {\bf 393} (2003) 159--164.

\bibitem{MZ} V. Moroz and J. Van Schaftingen, {\it Existence of groundstates for a class of nonlinear Choquard equations} Trans. Amer. Math. Soc. {\bf 367} (2015) 6557--6579.

\bibitem{PLJ2024} H. Pu, S. Liang, and S. Ji, {\it Nodal solutions to $(p, q)$-Laplacian equations with critical growth} Asymptot. Anal. {\bf 136} (2024) 133--156.

\bibitem{PS} P. Pucci and J. Serrin, {\it A general variational identity} Indiana Univ. Math. J. {\bf 35} (1986) 681--703.

\bibitem{articolopc_19} D. Ruiz, {\it The Schr\"odinger-Poisson equation under the effect of a nonlinear local term} J. Funct. Anal. {\bf 237} (2006) 655--674.

\bibitem{simon} J. Simon, {\it R\'egularit\'e de la solution d’une  \'equation non lin\'eaire dans $\mathbb R^N$} Lecture Notes in Math. Springer, Berlin {\bf 665} (1978) 205--227. 

\bibitem{SHR} Y. Song, Y. Huo, and D. D. Repov\v s, {\it On the Schr\"odinger-Poisson system with $(p, q)$-Laplacian} Appl. Math. Lett. {\bf 141} (2023).

\bibitem{SY} C. A. Swanson and L. S. Yu, {\it Critical $p$-Laplacian problems in $\mathbb R^N$} Ann. Mat. Pura Appl. {\bf 169} (1995), 233–250, 1995.

\bibitem{TZ23} M. Tao and B. Zhang, {\it Existence results for nonhomogeneous fractional Schr\"odinger-Poisson systems involving critical exponents}. Differential Integral Equations {\bf 36} (2023) 21--44.

\bibitem{Tolk} P. Tolksdorf, {\it Regularity for a more general class of quasilinear elliptic equations} J. Differential Equations {\bf 51} (1984) 126--150.

\bibitem{vaira} G. Vaira, {\it Ground states for Schr\"odinger-Poisson type systems} Ric. Mat. {\bf 60} (2011) 263--297.

\bibitem{WS25} L. Wei and Y. Song, {\it Normalized solutions for critical Schr\"odinger equations involving $(2, q)$-Laplacian} Opuscula Math. {\bf 45} (2015) 685--716.

\bibitem{Willemminimax} M. Willem, {\it Minimax theorems}  Progress in Nonlinear Differential Equations and their Applications {\bf 24}. Birkh\"auser Boston, Inc., Boston, MA, 1996.

\bibitem{CLW} C. Xiaoxiao, L. Anran, and W. Chongqing, {\it Solutions to non-homogeneous Schr\"odinger-Poisson system involving a $(p, q)$-laplacian operator} Journal of Applied Analysis and Computation {\bf 16} (2026) 1923--1950.

\bibitem{ZDS} J. Zhang, J. a. M. do O, and M. Squassina, {\it Fractional Schr\"odinger-Poisson systems with a general subcritical or critical nonlinearity} Adv. Nonlinear Stud. {\bf 16} (2016) 15--30.

\bibitem{ZZ2009} L. Zhao and F. Zhao, {\it Positive solutions for Schr\"odinger-Poisson equations with a critical exponent} Nonlinear Anal. {\bf 70} (2009) 2150--2164.

\bibitem{Zh86} V. V. Zhikov, {\it Averaging of functionals of the calculus of variations and elasticity theory} Izv. Akad. Nauk SSSR Ser. Mat. {\bf 50} (1986) 675--710.




\end{thebibliography}
\end{document}